\documentclass[11pt]{amsart}
\usepackage{mathtools}
\usepackage{amssymb}
\usepackage[margin=1in]{geometry}
\usepackage[utf8]{inputenc}
\usepackage[pdfdisplaydoctitle=true,
   colorlinks=true,
   urlcolor=blue,
   citecolor=blue,
   linkcolor=blue,
   pdfstartview=FitH,
   pdfpagemode= UseNone,
   bookmarksnumbered=true]{hyperref}
\newtheorem{theorem}{Theorem}
\newtheorem{corollary}[theorem]{Corollary}
\newtheorem{lemma}[theorem]{Lemma}
\newtheorem{proposition}[theorem]{Proposition}
\theoremstyle{definition}

\theoremstyle{remark}
\newtheorem{remark}[theorem]{Remark}

\newcommand{\RR}{\mathbb{R}}
\newcommand{\ZZ}{\mathbb{Z}}

\newcommand{\norm}[1]{\left\Vert#1\right\Vert}
\newcommand{\abs}[1]{\left\vert#1\right\vert}

\newcommand{\op}[1]{\mathbf{op}\left(#1\right)}

\DeclareMathOperator{\sgn}{sign} %
\DeclareMathOperator{\ad}{ad} %
\DeclareMathOperator{\Ad}{Ad} %
\DeclareMathOperator{\curl}{curl} %
\DeclareMathOperator{\diver}{div} %

\newcommand{\id}{\mathrm{id}}
\newcommand{\Rot}{\mathrm{Rot}}
\newcommand{\Diff}{\mathrm{Diff}}
\newcommand{\D}[1]{\mathcal{D}^{#1}(S^{1})}
\newcommand{\CS}{\mathrm{C}^{\infty}(\mathbb{S}^{1})}
\newcommand{\Vect}{\mathrm{Vect}}
\newcommand{\HH}[1]{H^{#1}(S^{1})}

\newcommand{\llangle}{\langle\!\langle}
\newcommand{\rrangle}{\rangle\!\rangle}

\mathtoolsset{showonlyrefs}
\hypersetup{
 pdftitle={Geometric investigations of a vorticity model equation}, 
 pdfauthor={Martin Bauer, Boris Kolev and Stephen C. Preston}, 
 pdfsubject={MSC 2010: 35Q35, 53D25}, 
  pdfkeywords={Generalized CLM equation, Fredholmness of the Riemannian exponential mapping, Sobolev metrics of fractional order}, 
 pdflang=EN 
 }
\begin{document}

\title{Geometric investigations of a vorticity model equation}%

\author[M. Bauer]{Martin Bauer}
\address{Faculty for Mathematics, University of Vienna, Austria}
\email{bauer.martin@univie.ac.at}

\author[B. Kolev]{Boris Kolev}
\address{Aix Marseille Universit\'{e}, CNRS, Centrale Marseille, I2M, UMR 7373, 13453 Marseille, France}
\email{boris.kolev@math.cnrs.fr}

\author[S.C. Preston]{Stephen C. Preston}
\address{Department of Mathematics, University of Colorado, Boulder, CO 80309-0395, USA}
\email{stephen.preston@colorado.edu}

\thanks{All authors were partially supported by the Erwin Schr\"{o}dinger Institute programme: Infinite-Dimensional Riemannian Geometry with Applications to Image Matching and Shape Analysis. M. Bauer was supported by the European Research Council (ERC), within the project 306445 (Isoperimetric Inequalities and Integral Geometry) and by the FWF-project P24625 (Geometry of Shape spaces). S.C. Preston was partially supported by Simons Collaboration Grant no. 318969.}%

\date{\today}%
\begin{abstract}
  This article consists of a detailed geometric study of the one-dimensional
  vorticity model equation
  \begin{equation}
    \omega_{t} + u\omega_{x} + 2\omega u_{x} = 0, \qquad \omega = H u_{x}, \qquad t\in\mathbb R,\; x\in S^{1}\,,
  \end{equation}
  which is a particular case of the generalized Constantin-Lax-Majda equation. Wunsch showed that this equation is the Euler-Arnold equation on $\operatorname{Diff}(S^{1})$ when the latter is endowed with the right-invariant homogeneous $\dot{H}^{1/2}$--metric. In this article we prove that the exponential map of this Riemannian metric is not Fredholm and that the sectional curvature is locally unbounded. Furthermore, we prove a Beale-Kato-Majda-type blow-up criterion, which we then use to demonstrate a link to our non-Fredholmness result. Finally, we extend a blow-up result of Castro-C\'{o}rdoba to the periodic case and to a much wider class of initial conditions, using a new generalization of an inequality for Hilbert transforms due to C\'{o}rdoba-C\'{o}rdoba.
\end{abstract}

\maketitle

\setcounter{tocdepth}{2}
\tableofcontents

\section*{Introduction}

In this paper we study geometric properties of the equation
\begin{equation}\label{main}
  \omega_{t} + u\omega_{x} + 2\omega u_{x} = 0, \qquad \omega = H u_{x}, \qquad t\in\RR,\; x\in S^{1}
\end{equation}
where $H$ is the Hilbert transform given by the integral transform
\begin{equation}\label{hilbertintegral}
  (H u)(x) = \frac{1}{2\pi} \, p.v. \int_{0}^{2\pi} u(y)\cot\left(\frac{x - y}{2}\right)\,dy
\end{equation}
or in terms of the Fourier transform as
\begin{equation}\label{hilbertfourier}
  (H u)(x) = -i\sum_{n\in\ZZ} \sgn(n) \hat{u}_{n}\, e^{in x},
\end{equation}
where $\hat{u}_{n}$ are the Fourier coefficients of $u$. Equation \eqref{main} is a particular case of the generalized Constantin-Lax-Majda equation introduced in~\cite{OSW2008} (corresponding to the parameter $a=-\frac{1}{2}$ in the notation of that paper). The original Constantin-Lax-Majda equation~\cite{CLM1985} is considered the simplest model of the blowup mechanism for the three-dimensional equations of ideal fluid mechanics, and related models are still being studied in this context~\cite{CHK+2014}. Wunsch~\cite{Wun2010} showed that equation \eqref{main} can be recast as the Euler-Arnold equation on $\Diff(S^{1})$ endowed with a right-invariant degenerate metric. For this reason we will refer in the remainder of the article to equation~\eqref{main} as the Wunsch equation.

To obtain the geometric interpretation of \eqref{main}, Wunsch equipped the space $\Diff(S^{1})$ with the right-invariant Riemannian metric given at the identity by
\begin{equation}\label{Hhalfdot}
  \llangle u, u\rrangle_{\dot{H}^{1/2}} = \int_{S^{1}} u H u_{x} \, dx = 2\pi \sum_{n\in \mathbb{Z}} \abs{n}\abs{\hat{u}_{n}}^{2} \quad \text{if} \quad u(x) = \sum_{n\in\mathbb{Z}} \hat{u}_{n} e^{inx}.
\end{equation}
Here the metric is both \emph{weak} in that it does not generate a topology in which $\Diff(S^{1})$ is a manifold, and \emph{degenerate} since constant vector fields have zero length in this metric. The latter can be avoided if we work on $\Diff(S^{1})/\Rot(S^{1})$, corresponding to the homogeneous space of diffeomorphisms modulo rotations. The local well-posedness of the geodesic equation induced by the homogeneous metric $\dot{H}^{1/2}$ on $\Diff(S^{1})/\Rot(S^{1})$ has been shown in~\cite{EKW2012}. Alternatively we can consider the group of diffeomorphisms fixing a single point.

The space $\Diff(S^1)/\Rot(S^1)$ with the $\dot{H}^{1/2}$-metric \eqref{Hhalfdot} has been studied geometrically by many authors, due to its appearance as a model for string theory~\cite{BR1987} and due to its structure as an infinite-dimensional K\"ahler manifold~\cite{KY1987}. It may be viewed as the space of densities on the circle, or as a coadjoint orbit of the Virasoro group, or as the space of univalent holomorphic functions in the disc, or as the space of closed Jordan curves in the complex plane; each of these pictures gives different geometric information. It is related to the space $\Diff(S^1)/PSL_2(\mathbb{R})$ with the $\dot{H}^{3/2}$-metric which arises in Teichm\"uller theory~\cite{NJ1990, GBR2015}. Formulas for sectional curvatures of the $\dot{H}^{1/2}$ and $\dot{H}^{3/2}$-metrics were computed using various methods in~\cite{BR1987, KY1987, Zum1988, GL2006}. See the book of Sergeev~\cite{Ser2010} for a recent survey of these and related topics. Our emphasis here is different since we will be
concerned primarily with the geodesic equation, which has mostly not been studied.

A nondegenerate metric on $\Diff(S^{1})$ with similar properties is denoted by $\mu H^{1/2}$ and given at the identity by
\begin{equation}\label{muHhalf}
  \llangle u, u\rrangle_{\mu H^{1/2}} = \int_{S^{1}} \left(\mu(u) + H u_{x}\right)u \, dx ,
\end{equation}
where
\begin{equation*}
  \mu(u) := \frac{1}{2\pi}\int_{S^{1}} u \, dx .
\end{equation*}
The Euler-Arnold equation is almost the same as \eqref{main}:
\begin{equation}\label{mainprime}
  \omega_{t} + u\omega_{x} + 2\omega u_{x} = 0, \qquad \omega = \mu(u) + H u_{x}, \qquad t\in\RR,\; x\in S^{1};
\end{equation}
only the definition of the momentum $\omega$ changes slightly. Even more, zero mean solutions of the above equation
are also solutions to the Wunsch equation \eqref{main}, as we show in Lemma~\ref{meanzerolemma}.
As an analogue of the $\mu$--Hunter--Saxton equation~\cite{KLM2008}, we will also refer to this equation as $\mu$--mCLM or
$\mu$--Wunsch equation. We may view the additional term in the $\mu H^{1/2}$-metric as analogous to the projection onto the
space of harmonic (constant) vector fields in the Euler equation on $\mathbb{T}^{2}$ or $\mathbb{T}^{3}$; the fact that mean-zero fields are preserved under the flow is analogous to the preservation of the average velocity of an ideal fluid~\cite{HTV2002}.

We will show that the $\mu H^{1/2}$--metric satisfies the conditions derived in~\cite{EK2014}.
Thus similarly as for the $\dot{H}^{1/2}$--metric the smoothness of the metric and the spray on the Sobolev completions follow. The main advantage in this situation is that we can work on the full diffeomorphism group $\Diff(S^{1})$ rather than on a subgroup or a homogeneous space. In addition we note for future reference that a very simple explicit solution of \eqref{mainprime} is $u(t,x) \equiv 1$ (so that $\omega(t,x)\equiv 1$), corresponding to unit-speed rotation of the circle. Note that there is no corresponding solution in the $\dot{H}^{1/2}$ case, and in fact it is unclear if there are any time-independent solutions of \eqref{main} at all for that case.

Another related metric is the full $H^{1/2}$-metric
\begin{equation}\label{Hhalf}
  \llangle u, u\rrangle_{{H}^{1/2}} = \int_{S^{1}} \left(u + H u_{x}\right)u \, dx\,,
\end{equation}
with corresponding Euler-Arnold equation
\begin{equation}\label{mainprimeprime}
  \omega_{t} + u\omega_{x} + 2\omega u_{x} = 0, \qquad \omega = u + H u_{x}, \qquad t\in\RR,\; x\in S^{1};
\end{equation}
Although we will focus our attention on the $\mu H^{1/2}$-metric, most of the results in this article continue to hold for
the full $H^{1/2}$-metric and for the degenerate $\dot{H}^{1/2}$-metric. Throughout the paper we will comment on the necessary
changes to deal with these related situations. The geometric picture for the $\dot{H}^{1/2}$-metric, which is slightly different from the other two cases, is separately discussed in Appendix~\ref{appendix:homogeneous-metric}.

Equations \eqref{main}, \eqref{mainprime}, and \eqref{mainprimeprime} are interesting geometrically since the Sobolev index $s=\frac{1}{2}$ of the corresponding metrics is critical for several important properties. For example, the Riemannian exponential map (which takes an initial velocity $u_{0}$ to $\eta(1)\in \Diff(S^{1})$, where $\eta$ is the flow defined by $\eta_{t} = u\circ\eta$ and $\eta(0)=\id$) is $C^{\infty}$ in any Sobolev completion $\Diff^{q}(S^{1})$ for $q > 3/2$ and also in the smooth category, and this only happens for the weak $H^{s}$-metric if $s\ge \tfrac{1}{2}$. On the other hand,~\cite{BBHM2013} proves that the geodesic distance (the infimum of lengths of smooth curves) vanishes in this metric, and this only happens for the weak $H^{s}$-metric if $s\le \tfrac{1}{2}$. The explanation for this paradoxical behavior is that although geodesics minimize between diffeomorphisms that are close in the strong topology ($H^{2}$ or stronger), there are highly oscillatory shortcuts that leave
small balls. Finally~\cite{MP2010} proved that the Riemannian exponential map is a nonlinear Fredholm map (that is,
its differential has finite-dimensional kernel and cokernel and closed range) for the right-invariant $H^{s}$-metric when $s>\tfrac{1}{2}$, but did not consider the critical case $s=\frac{1}{2}$.

\textbf{Contributions of the article:}
In this paper we show that the exponential maps for the $\dot{H}^{1/2}$, the $H^{1/2}$ and the $\mu H^{1/2}$-metrics are not Fredholm for essentially the same reason that Fredholmness fails for the $L^{2}$-metric on $\Diff_{\mu}(M^{3})$, the volume-preserving diffeomorphism group of a three-dimensional manifold~\cite{EMP2006}: there are linearly independent Jacobi fields $J_{n}$ along geodesics $\eta(t) = \exp_{\id}(tu_{0})$ with $J_{n}(0)=0$ and $J_{n}(t_{n})=0$, where $t_{n}$ is a sequence of times converging to some $T$; we use these to show that the range of $(d\exp_{\id})_{Tu_{0}}$ cannot be closed. In fact we can find conjugate points very easily
using a pointwise approximation of the Jacobi equation in much the same way as in~\cite{Pre2006}.
The results in this section will be formulated for the $\mu H^{1/2}$-metric only, but they continue to hold for the full $H^{1/2}$-metric and the $\dot{H}^{1/2}$-metric without any significant changes.

In Section~\ref{sec:blowup} we study the blow-up behavior of equations~\eqref{main} and \eqref{mainprime}. We prove a Beale-Kato-Majda-type blow-up criterion, which we then use to demonstrate a link to our non-Fredholmness result. Furthermore we extend a blow-up result of
Castro and C\'{o}rdoba~\cite{CC2010} to the periodic case and a wider class of initial conditions, using a special case of the new pointwise inequality
\begin{equation*}
  H(f H \Lambda^{p} f) + f\Lambda^{p}f \ge 0, \qquad \Lambda = H D_{x}
\end{equation*}
valid for any $f\colon S^{1}\to\RR$ and any $p>0$. This inequality extends a result of C\'{o}rdoba-C\'{o}rdoba~\cite{CC2003}, who proved the case $p=1$; we expect it to have many other applications in PDEs.

Finally, in Section~\ref{sec:curvature} we show that the sectional curvature for
the $\mu H^{1/2}$--metric admits both signs and is locally unbounded from above. It is unknown if the
$\dot{H}^{1/2}$-metric has any negative sectional curvature; we conjecture that it is positive as happens
for the $\dot{H}^{1}$-metric~\cite{Len2007}.

In Appendix~\ref{appendix:homogeneous-metric} we discuss the geometric picture for the homogeneous $\dot{H}^{1/2}$-metric on $S^1$, and in Appendix~\ref{appendix:real_line} we discuss all the metrics for the case of the diffeomorphism group on the non-compact manifold $\RR$.

\section{Geometric background material}

In this part we will study the $\mu H^{1/2}$-metric on the diffeomorphism group of the circle, while
recalling the major results for general Euler-Arnold equations.
For vector fields $u,v\in T_{\id}\Diff(S^{1})$ it is given by
\begin{equation}
  \llangle u, v\rrangle_{\mu H^{1/2}} = \int_{S^{1}} \left(\mu(u) + Hu_{x}\right)v \, dx\,.
\end{equation}
This inner product is then extended to a right invariant metric on all of $\Diff(S^{1})$ via right-translation:
\begin{equation}
  G^{1/2}_{\varphi}(X,Y) = \llangle X\circ\varphi^{-1}, Y\circ\varphi^{-1}\rrangle_{\mu H^{1/2}}\,.
\end{equation}

\subsection{The geodesic equation}

First recall that on $T_{\id}\Diff(S^{1})$ the Lie bracket is given by $ \ad_{u}v = vu_{x} - uv_{x}$ (the negative of the usual Lie bracket of vector fields, see Arnold-Khesin~\cite{AK1998}), while the group adjoint is given by
\begin{equation*}
  \Ad_{\eta}v = TL_{\eta}.TR_{\eta^{-1}}.v = (\eta_{x} v)\circ\eta^{-1}.
\end{equation*}

Let $\Lambda$ denote the first-order self-adjoint differential operator given by
\begin{equation}
  \Lambda u = \frac{1}{2\pi} \int_{S^{1}}u\, dx + H u_{x}
\end{equation}
so that the $\mu H^{1/2}$-metric can be written as
\begin{equation*}
  \llangle u, v\rrangle = \int_{S^{1}} u \Lambda v\, dx.
\end{equation*}
We first compute the operators $\ad^{\top}_{u}$ and $\Ad^{\top}_{\eta}$ from the Lie algebra $\mathfrak{g} = T_{\id}\Diff(S^{1})$
to itself. This computation appears in~\cite{MP2010,EK2014} and elsewhere; we repeat it here for the reader's convenience.

\begin{proposition}\label{starsprop}
  If $v\in T_{\id}\Diff(S^{1})$, then for any $u\in T_{\id}\Diff(S^{1})$ and $\eta\in \Diff(S^{1})$, we have
  \begin{equation}\label{adstar}
    \ad^{\top}_{u}v = \Lambda^{-1}(2u_{x} \Lambda v + u\Lambda v_{x})
  \end{equation}
  and
  \begin{equation}\label{Adstar}
    \Ad^{\top}_{\eta}v = \Lambda^{-1}\left[ \eta_{x}^{2} (\Lambda v)\circ\eta\right].
  \end{equation}
\end{proposition}

\begin{proof}
  Let $w$ be an arbitrary vector field on $S^{1}$. Then
  \begin{align*}
    \llangle \ad^{\top}_{u}v, w\rrangle & = \llangle v, \ad_{u}w\rrangle = \int_{S^{1}} (\Lambda v)(wu_{x} - uw_{x}) \, dx                                                                                           \\
                                        & = \int_{S^{1}} w\left[ u_{x} \Lambda v + \partial_{x}(u\Lambda v)\right] \, dx = \llangle w, \Lambda^{-1}\left[ u_{x} \Lambda v + \partial_{x}(u\Lambda v)\right]\rrangle, 
  \end{align*}
  because $\Lambda : C^{\infty}(S^{1}) \to C^{\infty}(S^{1})$ is a topological linear isomorphism, and we conclude \eqref{adstar}.
  Similarly we have
  \begin{align*}
    \llangle \Ad^{\top}_{\eta}v, w\rrangle & = \llangle v, \Ad_{\eta}w\rrangle = \int_{S^{1}} (\Lambda v) \eta_{x}\circ\eta^{-1} w\circ\eta^{-1} \, dx                                                  \\
                                           & = \int_{S^{1}} \left[ (\Lambda v)\circ\eta\right] \eta_{x}^{2} w \, dx = \llangle w, \Lambda^{-1}\left[ \eta_{x}^{2} (\Lambda v)\circ\eta\right] \rrangle, 
  \end{align*}
  and formula \eqref{Adstar} follows.
\end{proof}

Using the formula for $\ad^{\top}$, we obtain the geodesic equation:
\begin{corollary}
  The geodesic equation of the $\mu H^{1/2}$-metric on $\Diff(S^{1})$ is given by
  \begin{equation}\label{geod_equ_muH12}
    \eta_{t}\circ\eta^{-1}=u,\qquad \omega_{t} + u\omega_{x} + 2\omega u_{x} = 0, \qquad \omega = \frac{1}{2\pi}\int_{S^{1}} u\, dx+Hu_{x}.
  \end{equation}
\end{corollary}

\begin{remark}
  Let $\eta$ be the flow of the time-dependent vector field $u(t)$. Then
  \begin{equation*}
    \frac{d}{dt} \Ad^{\top}_{\eta}u = \Ad^{\top}_{\eta}\left(u_{t} + \ad^{\top}_{u}u\right).
  \end{equation*}
  Therefore, if $u$ is a solution of the Euler equation~\eqref{geod_equ_muH12}, then $\Ad^{\top}_{\eta(t)}u(t) = u(0)$ and we get
  \begin{equation}\label{conservation-law}
    \eta_{x}(t,x)^{2} \omega\left(t,\eta(t,x)\right) = \omega_{0}(x).
  \end{equation}
\end{remark}

\subsection{Connections to the Wunsch equation}

In the following we want to connect solutions to the geodesic equation of the $\mu H^{1/2}$-metric to solutions to the
Wunsch equation.
\begin{lemma}\label{meanzerolemma}
  The mean of the momentum $\int_{S^{1}} \omega(t,x)\,dx$ remains constant along any solution to~\eqref{geod_equ_muH12}.
  Thus any solution to~\eqref{geod_equ_muH12} that has zero initial mean velocity $\int_{S^{1}} u(0,x)\,dx = 0$ has
  zero mean velocity for all time $t$ and is therefore also a solution to the Wunsch equation.
\end{lemma}

A similar statement has been proven for the $\mu$-Hunter--Saxton equation (resp. the Hunter--Saxton equation) in~\cite{KLM2008} and more
generally this statement continues to hold for any pair of metrics $(\mu H^{s}, \dot{H}^{s})$ where $s>0$; here $\dot{H}^s$ generally denotes an inner product vanishing on constant vector fields and otherwise equivalent to $H^s$.

\begin{proof}
  Using the differential equation governing the evolution for $\omega$ we have
  \begin{equation*}
    \frac{d}{dt} \int_{S^{1}} \omega(t) = - \int_{S^{1}} u(t)\omega_{x}(t) + 2 u_{x}(t)\omega(t) = \int_{S^{1}} u(t)\omega_{x}(t) = \int_{S^{1}} u(t)\partial_{x}\Lambda u(t) = 0,
  \end{equation*}
  because the operator $\partial_{x}\Lambda$ is $L^{2}$ skew-symmetric. Now
  \begin{equation*}
    \int_{S^{1}} \omega(t) = \int_{S^{1}} \mu(u(t)) + \int_{S^{1}} H\partial_{x}u(t) = \int_{S^{1}} u(t).
  \end{equation*}
  Thus the mean of $u(t)$ is constant in time and the conclusion follows since~\eqref{geod_equ_muH12} reduces to the mCLM equation if the mean of $u(0)$ vanishes.
\end{proof}

\subsection{Local well-posedness}
\label{subsec:local-well-posedness}

There are two ways to solve equation~\eqref{muHhalf}; the first one uses the second-order spray method of Ebin-Marsden~\cite{EM1970}, while the other one is based on the following lemma due to Ebin~\cite{Ebi1984}. (See also Majda-Bertozzi~\cite{MB2002}, where the latter is described as the ``particle-trajectory method.'')

\begin{lemma}
  Let $G$ be a Lie group with a right-invariant metric. A smooth curve $\eta(t)$ is a geodesic issued from the neutral element $e$ with initial velocity $u_{0}$ iff $\eta(t)$ is an integral curve of the vector field
  \begin{equation}\label{particule-trajectory}
    X(\eta) := \left(TL_{\eta^{-1}}\right)^{\top}u_{0}.
  \end{equation}
\end{lemma}

\begin{proof}
  Let $\eta(t)$ be a a geodesic, then we have
  \begin{equation}
    \frac{d\eta}{dt} = TR_{\eta}u
  \end{equation}
  where the Eulerian velocity $u$ satisfies the Euler-Arnold equation
  \begin{equation*}
    \frac{du}{dt} = -B(u,u)
  \end{equation*}
  and $B(u,v) := \frac{1}{2}\left(\ad^{\top}_{u}v + \ad^{\top}_{v}u\right)$ is the Arnold operator. The Euler-Arnold equation implies
  \begin{equation*}
    \frac{d}{dt} \left( \Ad_{\eta(t)}^{\top}u(t)\right) = 0,
  \end{equation*}
  so that $u(t) = \Ad_{\eta^{-1}}^{\top}u_{0}$.
  Therefore, the flow equation becomes
  \begin{equation*}
    \frac{d\eta}{dt} = \left(TR_{\eta}\right)\left(\Ad_{\eta^{-1}}^{\top}\right)u_{0} = \left(TR_{\eta}\right) \left(TR_{\eta}\right)^{\top} \left(TL_{\eta^{-1}}\right)^{\top}u_{0} = \left(TL_{\eta^{-1}}\right)^{\top}u_{0},
  \end{equation*}
  because $R_{\eta}$ is a Riemannian isometry and thus $\left(TR_{\eta}\right)\left(TR_{\eta}\right)^{\top} = \id_{T_{\eta}G}$.
\end{proof}

Consider now the special case where $G=\Diff(S^{1})$ is the Fr\'{e}chet Lie group of orientation-preserving diffeomorphisms of the circle, with a right-invariant metric induced by an inertia operator $A$:
\begin{equation*}
  \llangle u,w\rrangle_{\id} := \int_{S^{1}} (Au)w \, dx, \qquad u,w \in \Vect(S^{1}),
\end{equation*}
where $A : C^{\infty}(S^{1}) \to C^{\infty}(S^{1})$ is a $L^{2}$-symmetric \emph{Fourier multiplier}. That is $A$ is continuous and commutes with $\partial_{x}$.

Starting with the Euler equation
\begin{equation*}
  u_{t} = -\ad_u^{\top}u = - A^{-1}\left(2u_{x}Au + u(Au)_{x}\right),
\end{equation*}
we compute the second order spray as follows. Let $\eta(t)$ be a geodesic and $v(t) :=\eta_{t}(t)$ so that $v = u \circ \eta$. We get
\begin{equation*}
  v_{t} = u_{t} \circ \eta + (u_{x}\circ \eta)(u \circ \eta)
\end{equation*}
which leads to
\begin{equation*}
  v_{t} = \left\{A^{-1}\Big(-2u_{x}Au - u(Au)_{x} + A(uu_{x})\Big)\right\}\circ \eta = \left\{ A^{-1}\Big([A,uD]u-2(Au)u_{x}\Big)\right\}\circ \eta.
\end{equation*}
The geodesic spray can thus be written as:
\begin{equation}\label{eq:P1}
  \left\{\begin{aligned}
  \eta_{t} &= v, \\
  v_{t} &= S_{\eta}(v),
  \end{aligned}
  \right.
\end{equation}
where
\begin{equation*}
  S_{\eta}(v):=\left(TR_{\eta}\circ S\circ TR_{\eta^{-1}} \right)(v),
\end{equation*}
and
\begin{equation*}
  S(u):= A^{-1}\left\{ [A,uD]u-2(Au)(Du)\right\},
\end{equation*}
with initial conditions $\eta(0) = \id$ and $v(0)=u_{0}$.

On the other way, we have
\begin{equation*}
  TL_{\eta^{-1}} v = \frac{1}{\eta_{x}} v.
\end{equation*}
where $\eta \in \Diff(S^{1})$ and $v \in T_{\eta}\Diff(S^{1})$. Therefore, for every $u \in \Vect(S^{1})$, we have
\begin{equation*}
  \llangle TL_{\eta^{-1}} v,u\rrangle_{\id} = \llangle v,\left(TL_{\eta^{-1}}\right)^{\top} u\rrangle_{\eta}
\end{equation*}
and we get
\begin{equation*}
  \left(TL_{\eta^{-1}}\right)^{\top} u = A_{\eta}^{-1}\left(\frac{1}{\eta_{x}^{2}}Au\right),
\end{equation*}
where
\begin{equation*}
  A_{\eta} := TR_{\eta} \circ A \circ TR_{\eta^{-1}}.
\end{equation*}
Ebin's reformulation leads thus to the first order Cauchy problem
\begin{equation}\label{eq:P2}
  \eta_{t} = A_{\eta}^{-1}\left(\frac{1}{\eta_{x}^{2}}Au_{0}\right),
\end{equation}
with initial condition $\eta(0) = \id$. We can write \eqref{eq:P2} more
explicitly as an ODE for $\rho = \eta_{x}$ involving an integral transform,
but we will not pursue this since it is not necessary here.

In both cases, we need to show that these vector fields extend smoothly on some Hilbert approximation manifolds $\D{q}$ of $\Diff(S^{1})$, that we review briefly first. Let $\HH{q}$ be the completion of $C^{\infty}(S^{1})$ for the norm
\begin{equation*}
  \norm{u}_{H^{q}}:= \left( \sum_{k \in \ZZ}(1 + k^{2})^{q}\abs{\hat{u}_{k}}^{2} \right)^{1/2},
\end{equation*}
where $q\in \RR$ and $q \ge 0$. A $C^{1}$ diffeomorphism $\eta$ of $S^{1}$ is of class $H^{q}$ if for any lift $\tilde{\eta}$ to $\RR$, we have
\begin{equation*}
  \tilde{\eta} - \id \in \HH{q}.
\end{equation*}

For $q > 3/2$, the set $\D{q}$ of $C^{1}$-diffeomorphisms of the circle which are of class $H^{q}$ has the structure of a \emph{Hilbert manifold}, modeled on $\HH{q}$ (see~\cite{EM1970} or~\cite{IKT2013}). The manifold $\D{q}$ is also a \emph{topological group} but \emph{not a Lie group} (composition and inversion in $\D{q}$ are continuous but \emph{not differentiable}). Note however, that, given $\varphi \in \D{q}$,
\begin{equation*}
  u \mapsto R_{\varphi}(u) := u \circ \varphi, \qquad \HH{q} \to \HH{q}
\end{equation*}
is a \emph{smooth map}, and that
\begin{equation*}
  (u,\varphi) \mapsto u \circ \varphi, \qquad \HH{q+k} \times \D{q} \to \HH{q}
\end{equation*}
is of class $C^{k}$.

Coming back to our problem, we have to show either that that~\eqref{eq:P1} extends to a well-defined and smooth second order vector field on $\D{q}$, or that~\eqref{eq:P2} extends to a well-defined and smooth vector field on $\D{q}$. Both problems require some kind of \emph{ellipticity condition}, namely that $A$ extends to a topological linear isomorphism from $\HH{q}$ to $\HH{q-r}$ for some $r \ge 1$. In the sequel, we will assume that this \emph{ellipticity condition} is fulfilled and take the index $q$ such that $q > 3/2$ and $q \ge r$. Under these hypotheses, both vector fields are well defined on $\D{q}$. The main problem is to show that they are smooth.

It was shown in~\cite[Theorem 3.10]{EK2014} that the smoothness of the spray~\eqref{eq:P1} reduces to show the smoothness of the mapping
\begin{equation}\label{operator-smoothness}
  \eta \mapsto A_{\eta}, \qquad \D{q} \to \mathcal{L}(\HH{q},\HH{q-r}).
\end{equation}

Concerning the second problem~\eqref{eq:P2}, and given our choice of indices $q,r$, the mapping
\begin{equation*}
  \eta \mapsto \frac{1}{\eta_{x}^{2}} Au_{0}, \qquad \D{q} \to \HH{q-r}
\end{equation*}
is smooth. Therefore, the smoothness of Ebin's vector field
\begin{equation*}
  \eta \mapsto A_{\eta}^{-1}\left(\frac{1}{\eta_{x}^{2}}Au_{0}\right)
\end{equation*}
reduces to show that the mapping
\begin{equation}\label{two-component-smoothness}
  (\eta,w) \mapsto A^{-1}_{\eta}(w), \qquad \D{q}\times\HH{q-r} \to \HH{q}
\end{equation}
is smooth, which is the same as showing the smoothness of~\eqref{operator-smoothness} (see~\cite{EK2014}).

Finally, if $A$ is a differential operator with constant coefficients, it is just an exercise to show that the mapping
\begin{equation*}
  \eta \mapsto A_{\eta}, \qquad \D{q} \to \mathcal{L}(\HH{q},\HH{q-r}),
\end{equation*}
is smooth (indeed real analytic). It was the goal of~\cite{EK2014} to show that this result extends for a more general class of \emph{non-local inertia operators}. If $A$ is a Fourier multiplier, we have
\begin{equation}\label{fouriermultiplier}
  (Au)(x) = \sum_{n \in \mathcal Z} a(n)\hat{u}_{n} e^{i n x},
\end{equation}
where $a : \ZZ \to \RR$ is a real even function, called the \emph{symbol} of $A$. We will write $A = a(D)$ or $A = \op{a(\xi)}$. A Fourier multiplier $A = a(D)$ is of class $\mathcal{S}^{r}$ if $a$ is the restriction to $\ZZ$ of a smooth function such that
\begin{equation}\label{eq:Sr-condition}
  \abs{a^{(n)}(\xi)} \lesssim (1 + \abs{\xi}^{2})^{(r-n)/2}, \qquad \forall n \in \mathbb N.
\end{equation}

It was shown in~\cite{EK2014} that the mapping~\eqref{operator-smoothness}~is smooth, provided that $A$ is of class $\mathcal{S}^{r}$ and $r \ge 1$. Moreover, this result is still true if $a$ has some singularities at $\xi = 0$ like $\abs{\xi}$ or contains an additional term like $\delta_{0}(\xi)$.

As a consequence of this we obtain the following result concerning smoothness of the metric and the spray.

\begin{theorem}\label{thm:local-wellposedness}
  The $\mu H^{1/2}$-metric and its geodesic spray extend smoothly to the Hilbert approximation manifolds $\D{q}$ for all $q>\frac32$.
  The corresponding exponential map is a smooth local diffeomorphism from a neighborhood $V$ of $0$ onto a neighborhood $U$ of the identity in $\D{q}$.
\end{theorem}

\begin{remark}
  A similar statement holds for the full $H^{1/2}$-metric
  \begin{equation}
    \llangle u, u\rrangle_{H^{1/2}} = \int_{S^{1}} u^{2} \,dx + \int_{S^{1}} u Hu_{x} \, dx,
  \end{equation}
  and also for the degenerate $\dot{H}^{1/2}$-metric --- here one has to consider the degenerate metric as a metric
  on the subgroup of diffeomorphisms that preserve one point; see Appendix~\ref{appendix:homogeneous-metric}.
\end{remark}

\begin{proof}
  The $\mu H^{1/2}$-metric corresponds to the inertia operator
  \begin{equation*}
    \Lambda := \op{\delta_{0}(\xi) + \abs{\xi}}.
  \end{equation*}
  This operator fulfills the above assumptions and thus we can apply the results of~\cite{EK2014}.
\end{proof}

\subsection{The induced geodesic distance}
\label{sec:geodesic_distance}

The induced geodesic distance of a Riemannian metric $G$ is defined as the infimum of the lengths
of all paths that connect two given points:
\begin{equation}
  \operatorname{dist}^G(\varphi_{1},\varphi_2)=\operatorname{inf} \int_{0}^{1} \sqrt{G_{\varphi}(\varphi_{t},\varphi_{t})}\,,
\end{equation}
where the infimum is taken over all paths $\varphi: [0,1]\to \operatorname{Diff}(S^{1})$ with
$\varphi(0)=\varphi_{1}$ and $\varphi(1)=\varphi_2$.
It was very surprising when Michor and Mumford proved in~\cite{MM2005} that the right invariant
$L^{2}$-metric on $\Diff(S^{1})$ induces vanishing geodesic distance, i.e., any two diffeomorphisms $\varphi_{1}, \varphi_2\in \Diff(S^{1})$ can be connected by paths of arbitrary short length.
This result was extended to the class of Sobolev-type metrics of order $s\leq\frac12$
in~\cite{BBHM2013,BBM2013}. It turns out that this phenomenon also occurs for the $\mu H^{1/2}$-metric.

\begin{theorem}
  The geodesic distance for the $\mu H^{1/2}$-metric on $\Diff(S^{1})$ vanishes.
\end{theorem}

The proof for the homogeneous $\dot{H}^{1/2}$-metric is discussed in Appendix~\ref{appendix:homogeneous-metric}.

\begin{proof}
  The result follows from the fact that on a compact manifold the $\mu H^{1/2}$-norm is bounded from above by the $H^{1/2}$-norm
  \begin{equation*}
    \llangle u, u\rrangle_{\mu H^{1/2}} = \sum_{n \in \ZZ} (\delta_{0}(n) + \abs{n}) \abs{\hat{u}_{n}}^{2} \leq \sum_{n \in \ZZ} (1 + \abs{n}) \abs{\hat{u}_{n}}^{2} = \llangle u, u\rrangle_{H^{1/2}}
  \end{equation*}
  Now vanishing geodesic distance for the $\mu H^{1/2}$-metric follows from the vanishing geodesic distance result for the full $H^{1/2}$-metric.
  For the convenience of the reader we will sketch the main arguments of the proof given in~\cite{BBHM2013} in our slightly different situation.
  The proof can be split in the following two steps:
  
  \begin{enumerate}
    \item The set $\Diff(S^{1})^{L=0}$ of all diffeomorphism that can be connected from the identity by paths of arbitrary short length is a normal subgroup of the diffeomorphism group.
    \item $\Diff(S^{1})^{L=0}$ is non-trivial, i.e., there exists a diffeomorphism $\varphi\neq \id$ with $\varphi\in \Diff(S^{1})^{L=0}$.
  \end{enumerate}
  
  Then the statement follows using that $\Diff(S^{1})$ is a simple group, i.e., it has no nontrivial normal subgroups and thus $\Diff^{L=0}(S^{1})$
  needs to be equal to the whole connected component of the identity.
  
  To prove the first statement let $\psi_{1} \in \Diff(S^{1})$ and $\varphi_{1} \in \Diff(S^{1})^{L=0}$.
  Now we consider a curve
  $t \mapsto \varphi(t,\cdot)$ from the identity to $\varphi_{1}$ with length less than
  $\epsilon$. Then $\psi(t):=\psi_{1} \circ \varphi(t) \circ \psi_{1}^{-1}$ connects the identity
  to $\psi_{1} \circ \varphi_{1} \circ \psi_{1}^{-1}$. We will now show that the length of $\psi(t)$
  is smaller than a constant times $\epsilon$. Therefore we let $u=\varphi_{t} \circ \varphi^{-1}$. Then we calculate
  \begin{align*}
    \operatorname{Len}(\psi_{1}^{-1} \circ \varphi \circ \psi_{1} ) & = \int_{0}^{1} \norm{(\psi_{1}^{\prime} \circ \psi_{1}^{-1}) \cdot (\varphi_{t} \circ \varphi \circ \psi_{1})^{-1}}_{\mu H^{1/2}} dt \\ &=
    \int_{0}^{1} \norm{\psi_{1}^{\prime} \circ \psi_{1}^{-1} \cdot u\circ\psi_{1}^{-1}}_{\mu H^{1/2}} dt\leq
    C(\psi_{1}) \int_{0}^{1} \norm{u}_{\mu H^{1/2}} dt
    \\&= C(\psi_{1})\operatorname{Len}(\varphi)\leq C(\psi_{1})\epsilon\;.
  \end{align*}
  Here we used that for $h\in C^{\infty}(S^{1})$ and $\psi_{1} \in \Diff(S^{1})$ pointwise multiplication
  $f\mapsto h\cdot f$ and composition $f\mapsto f\circ\psi $ are bounded linear operators for the
  $\mu H^{1/2}$-norm. Thus we have seen that $\Diff(S^{1})^{L=0}$ is a normal subset of $\Diff(S^{1})$. It remains to prove that there exists one non-trivial diffeomorphism in $\Diff(S^{1})^{L=0}$. To do this we will construct arbitrary short paths from the identity to the shift $\varphi_{1}(x) = x+1$. To define these paths we will need functions that have small $\mu H^{1/2}$-norm and large $L^{\infty}$-norm at the same time. For the full $H^{1/2}$-norm the existence of such functions has been proven in~\cite{Tri2001}.
  Since we can bound the $H^{1/2}$-norm by the $\mu H^{1/2}$-norm a similar statement holds in our situation.
  In particular this result provides us with a family of functions $f_{N}$ such that
  $\norm{f_{N}}^{2}_{\mu H^{1/2}} \leq C\frac 1N$, while $\norm{f_{N}}_\infty = 1$.
  
  Now we can define a family of vector fields via
  \begin{equation*}
    u_{N}(t,x)=\lambda f_{N}(t-x) \quad\text{with}\quad 0 \le \lambda < 1
  \end{equation*}
  for $t \in [0, T_{\operatorname{end}}]$.
  The timepoint $T_{\operatorname{end}}$ has to be chosen such that the flow $\varphi_{N}$ of $u_{N}$ satisfies
  $\varphi_{N}(T_{\operatorname{end}},x) = x+1$. To show that the flow of this vector field
  at time $t=T_{\operatorname{end}}$ is indeed $\varphi(T_{\operatorname{end}},x) = x+1$ and that $T_{\operatorname{end}}$ does not grow as $N\rightarrow \infty$, we refer to the proof
  in~\cite{BBHM2013}.
  
  Now the energy of this path is bounded by
  \begin{equation}
    E(u_{N}) = \int_{0}^{T_{\operatorname{end}}} \norm{ u(t,.)}^{2}_{\mu H^{1/2}} dt \leq C\, T_{\operatorname{end}} \frac 1N
  \end{equation}
  and thus the result follows.
\end{proof}

\begin{remark}
  We have now seen that both the full $H^{1/2}$-metric and the $\mu H^{1/2}$-metric induce vanishing geodesic distance. On the other hand, we have shown in the previous
  section that the geodesic flow for both metrics is $C^{\infty}$ in any Sobolev completion $\D{q}$ for $q > 3/2$ and as a consequence
  that the corresponding exponential map is a smooth local diffeomorphism from a neighborhood of $0$ onto a neighborhood of the identity in $\D{q}$. By the Gau{\ss} lemma this result seems to guarantee positivity of the geodesic distance.
  This paradoxical behavior can be explained by the observation that although geodesics minimize between diffeomorphisms that are close in the strong topology ($H^{2}$ or stronger), there are highly oscillatory shortcuts that leave small strong balls (but stay in small weak balls).
\end{remark}

For the $H^{s}$-metric on $\Diff(S^{1})$ with $s<\frac12$, the article~\cite{BBHM2013} provides a proof of the vanishing geodesic distance result that does not need to use the simplicity
of the diffeomorphism group, i.e., for any diffeomorphism $\varphi_{1}\in \operatorname{Diff}(S^{1})$ they construct paths from the identity to $\varphi_{1}$ with arbitrary short length.
The reason for this is that the indicator function of an interval $[a,b]$ is an element of $H^{s}$ if and only if $s<\frac12$. Defining the vector fields $u_{N}$ with indicator functions --- instead of
the functions $f_{n}$ as defined above ---
allows much more precise control on the endpoint of the corresponding flows. It would be interesting to generalize this direct proof to the case $s=\frac12$, since this would then also yield a proof for
the diffeomorphism group on the non-compact manifold $\RR$; see~\cite{BBHM2013,BBM2013} and Appendix~\ref{appendix:real_line} for a discussion on these questions.

\section{Non-Fredholmness of the exponential map}
\label{sec:Fredholmness}

We have already seen that the $\mu H^{1/2}$-metric is geometrically interesting since the Sobolev index $s=\frac{1}{2}$ is critical for several important properties. In this part we will study Fredholm properties of the exponential map. In~\cite{MP2010} it was proved that the Riemannian exponential map is a nonlinear Fredholm map for the right-invariant $H^{s}$-metric when $s>\tfrac{1}{2}$, but the authors did not consider the critical case $s=\frac{1}{2}$. In the following we will show that for the critical value, the exponential map fails to be Fredholm. In order to prove this result, we will construct linearly independent Jacobi fields $J_{n}$ along geodesics $\eta(t) = \exp_{\id}(tu_{0})$ with $J_{n}(0)=0$ and $J_{n}(t_{n})=0$, where $t_{n}$ is a sequence of times converging to some $T$; we use these to show that the range of $(d\exp_{\id})_{Tu_{0}}$ cannot be closed.

Therefore we will use the fact from general Riemannian geometry that there is a Jacobi field vanishing at times $t=a$ and $t<b$ if and only if there is a vector field $v(t)$ such that $v(a)=v(b)=0$ and $I(v,v)<0$, where $I$ is the index form. (This fact depends only on properties of second-order self-adjoint differential equations and is thus true for both weak and strong metrics in finite dimensions or infinite dimensions.)

On any Riemannian manifold the Morse index form corresponding to the Jacobi equation
\begin{equation}
  \frac{D^{2}J}{dt^{2}} + R(J,\dot{\eta})\dot{\eta}=0
\end{equation}
is given by
\begin{equation}\label{Indexform}
  I(J,J) = \int_{a}^{b} \norm{\frac{DJ}{dt}}^{2} - \llangle R(J,\dot{\eta})\dot{\eta}, J\rrangle \, dt;
\end{equation}
that is, we compute the dot product of the Jacobi equation with a Jacobi field and integrate by parts. Using the formulas from~\cite{MP2010}, we can compute a much simpler version of the index form. The formula involves only left translations, since for a right-invariant metric all the geometry is determined by the left translations (otherwise the metric would be bi-invariant, and the Riemannian exponential map would reduce to the group exponential map).

\begin{lemma}\label{indexformlemma}
  Suppose $G$ is a Lie group with a weak right-invariant metric $\llangle \cdot, \cdot\rrangle$, and let $\eta$ be a geodesic in $G$
  with $\eta(0)=\id$ and $\dot{\eta}(0)=u_{0}$, defined on $[0,T]$.
  Then for any $a,b$ with $0\le a<b\le T$, the Morse index form for a Jacobi field $J = TL_{\eta}v$ is given by
  \begin{equation}\label{indexform}
    I(J,J) = \int_{a}^{b} \norm{\Ad_{\eta(t)} \dot{v}(t)}^{2} + \llangle u_{0}, \ad_{v(t)}\dot{v}(t)\rrangle \, dt.
  \end{equation}
  If the index form is negative for some field $v$ with $v(a)=v(b)=0$, then $\eta(a)$ is monoconjugate to $\eta(b-\varepsilon)$ for some $\varepsilon>0$,
  in the sense that there is a Jacobi field with $J(a)=J(b-\varepsilon)=0$.
\end{lemma}

\begin{remark}
  Note, that for the degenerate $\dot{H}^{1/2}$-metric the operator $\ad_{u}^{\top}v$ does not exist in general, but only the symmetric version $\beta(u,v)=\frac12 (\ad_{u}^{\top}v+\ad_{v}^{\top}u)$. However, one can reformulate the above lemma in terms of $\beta$. See Appendix~\ref{appendix:homogeneous-metric} for more comments on this situation.
  In the following we will focus our analysis on the non-degenerate $\mu H^{1/2}$-metric and therefore we will not worry about this.
\end{remark}

\begin{proof}
  Given a path $\eta(t)$ in $G$, the covariant derivative of a vector field $X(t)$ defined along the path may be written as
  \begin{equation*}
    \frac{DX}{dt} = TR_{\eta} \left\{ \frac{dv}{dt} + \nabla_{u}v \right\}.
  \end{equation*}
  where $u(t) := TR_{\eta^{-1}}\frac{d\eta}{dt}$, $v(t) := TR_{\eta^{-1}}X(t)$ and $\nabla_{u}v := \frac{1}{2}\left(-\ad_uv + \ad^{\top}_{u}v + \ad^{\top}_{v}u\right)$
  by the Levi-Civita formula for right-invariant vector fields. Here we are using $\ad_uv = -[u,v]$ where the latter denotes the usual Lie bracket of vector fields; see
  Arnold-Khesin~\cite{AK1998}.
  
  The index form is given by \eqref{Indexform}. Write $J = TR_{\eta}y$ for $y\in T_{\id}G$, and we have
  $\llangle R(J,\dot{\eta})\dot{\eta},J\rrangle = \llangle R(y,u)u,y\rrangle$ by right-invariance of both
  the metric and the curvature. Furthermore, the definition of the curvature gives
  \begin{equation*}
    \llangle R(y,u)u,y\rrangle = \llangle \nabla_y\nabla_{u}u,y\rrangle - \llangle \nabla_{u}\nabla_yu,y\rrangle + \llangle \nabla_{[u,y]}u,y\rrangle.
  \end{equation*}
  Now set $z=\frac{dy}{dt} - \ad_uy$, so that $\frac{dy}{dt} + \nabla_{u}y = z + \nabla_yu$. Using the Euler
  equation $\frac{du}{dt} + \nabla_{u}u=0$, the index form becomes
  \begin{align*}
    I(J,J) & = \int_{a}^{b} \norm{z + \nabla_yu}^{2} + \llangle \nabla_y\tfrac{du}{dt},y\rrangle - \llangle \nabla_yu, \nabla_{u}y\rrangle - \llangle \nabla_{[u,y]}u,y\rrangle \, dt                                              \\
           & = \int_{a}^{b} \norm{z}^{2} + 2\llangle z,\nabla_yu\rrangle - \llangle \nabla_{z-[u,y]}u,y\rrangle - \llangle \nabla_yu,z-[u,y]\rrangle - \llangle [u,y],\nabla_yu\rrangle - \llangle \nabla_{[u,y]}u,y\rrangle \, dt \\
           & = \int_{a}^{b} \norm{z}^{2} + \llangle z,\nabla_yu\rrangle - \llangle y,\nabla_zu\rrangle \, dt                                                                                                                       \\
           & = \int_{a}^{b} \norm{z}^{2} - \llangle u, [y,z]\rrangle \, dt.                                                                                                                                                        
  \end{align*}
  
  Finally we use the fact that the Euler equation $\frac{du}{dt} + \ad^{\top}_{u}u = 0$ implies that $u=\Ad^{\top}_{\eta^{-1}}u_{0}$, so
  that the last term becomes
  \begin{equation*}
    \llangle u, [y,z] \rrangle = \llangle u_{0}, \Ad_{\eta^{-1}}[y,z]\rrangle = \llangle u_{0}, [\Ad_{\eta^{-1}}y, \Ad_{\eta^{-1}}z]\rrangle.
  \end{equation*}
  Now set $y = \Ad_{\eta}v$ and $z=\Ad_{\eta}w$; then the fact that $z = \frac{dy}{dt} - \ad_{u}y$ implies that
  $ w = \frac{dv}{dt}$, so $z = \Ad_{\eta}\dot{v}$. This finally gives the formula \eqref{indexform}.
\end{proof}

\subsection{Explicit solution of the Jacobi equation}

The easiest way to prove that the $\mu H^{1/2}$ exponential map on $\Diff(S^{1})$ is not Fredholm is to
compute all the Jacobi fields explicitly along a particular geodesic. Recall from the Introduction
that if $u_{0}$ is the constant vector field $u_{0}(x)\equiv 1$, then $u(t,x) \equiv 1$ for all time.
In~\cite{MP2010} it was shown that on any Lie group with a right-invariant metric,
the Jacobi equation for a Jacobi field $J(t) = (d\exp_{\id})_{tu_{0}}(tw_{0})$ along a geodesic $\eta(t) = \exp_{\id}(tu_{0})$ can be written in terms of the left translation $J(t) = dL_{\eta(t)}v(t)$ as the system
\begin{equation}\label{jacobisystem}
  \frac{dv}{dt} = w, \qquad \frac{d}{dt}\left( \Ad^{\top}_{\eta(t)}\Ad_{\eta(t)} w(t)\right) + \ad^{\top}_{w(t)}u_{0} = 0, \qquad v(0)=0, \quad w(0)=w_{0}.
\end{equation}

Solving this equation along a particularly simple geodesic immediately lets us prove the following theorem,
analogous to the result in~\cite{EMP2006} for the three-dimensional volumorphism group.

\begin{theorem}\label{nonfredexample}
  The exponential map for the $\mu H^{1/2}$-metric on $\Diff(S^{1})$ is not Fredholm.
\end{theorem}

This result can be proven in a similar manner for the full $H^{1/2}$-metric, but not for the degenerate $\dot{H}^{1/2}$-metric, since the basic ingredient --- the existence and explicit formula of the steady state solutions --- is not known in this case.

\begin{proof}
  Let $u_{0}\equiv 1$. Then $u(t,x) = 1$ for all $t$ and $x$, so that the flow $\eta$ is given by $\eta(t,x) = x+t$ (modulo $2\pi$).
  Then $\Ad_{\eta(t)}$ is an isometry, so we see that $\Ad^{\top}_{\eta(t)}\Ad_{\eta(t)}$ is the identity for
  all $t$. On the other hand by formula \eqref{adstar} we have
  \begin{equation}
    \ad^{\top}_wu_{0} = \Lambda^{-1}(2w_{x} \Lambda u_{0} + w\Lambda (u_{0})_{x}) = 2\Lambda^{-1}(w_{x}).
  \end{equation}
  where $\Lambda$ is the corresponding inertia operator of the $\mu H^{1/2}$-metric.
  Now expand $w$ in a Fourier series as
  \begin{equation*}
    w(t) = \sum_{n\in\mathbb{Z}} w_{n}(t) e^{inx}, \qquad w_{0} = \sum_{n\in\mathbb{Z}} c_{n} e^{inx},
  \end{equation*}
  and \eqref{jacobisystem} becomes
  \begin{equation*}
    \frac{dw_{n}}{dt} + \frac{2in w_{n}(t)}{\delta_{0}(n)+ \abs{n}} = 0, \qquad w_{n}(0) = c_{n},
  \end{equation*}
  with solution
  \begin{equation*}
    w_{n}(t) = c_{n} \exp\left(\frac{-2int}{\delta_{0}(n)+\abs{n}}\right).
  \end{equation*}
  Solving for $v(t,x) = \sum_{n\in\mathbb{Z}} v_{n}(t) e^{inx}$ with initial condition $v_{n}(0)=0$, we obtain
  \begin{equation*}
    v_{n}(t) =
    \begin{cases} c_n e^{-it\sgn{n}} \sin{t} & n\ne 0 \\
      c_0 t & n=0 
    \end{cases}
  \end{equation*}
  so that $v_{n}(\pi)=0$ whenever $n\ne 0$. We thus see that $\eta(\pi)$ is
  monoconjugate to $\eta(0)$ of infinite order, and so the nullspace of $(d\exp_{\id})_{\pi u_0}$ is infinite-dimensional,
  while its image is just the space of constant vector fields and thus its codimension is also infinite. But a Fredholm operator
  must have finite-dimensional kernel and cokernel by definition.
\end{proof}

\subsection{Conjugate points along arbitrary geodesics}\label{sec:conjugate_{a}rbitrary}

In the previous section we have constructed families of conjugate points along very special solutions to the geodesic equation. In the following we want to show that this is in fact not an isolated phenomenon, but that these families of conjugate points continue to exist along arbitrary geodesics. First we need to estimate the term $\norm{\Ad_{\eta}v}^{2}_{\mu H^{1/2}}$. We use the Cauchy-Schwarz inequality to estimate in terms of the $H^{1}$-norm, but note that we have to do this somewhat delicately to get everything in terms of $\eta_{x}$ alone in Theorem~\ref{conjugateexistence}.

\begin{lemma}\label{adadestimatelemma}
  Suppose $\eta\in \Diff(S^{1})$ is smooth. Then for any vector field $v$ and any $\varepsilon>0$ we have
  \begin{equation}\label{adadestimate}
    \norm{\Ad_{\eta}v}^{2}_{\mu H^{1/2}}
    \le \int_{S^{1}} \left[ \tfrac{1}{2\varepsilon} \phi^{2} + \tfrac{\varepsilon}{2} \phi_{x}^{2} \right]\,dx
    + \frac{1}{2\pi} \left(\int_{S^{1}} \eta_{x}\phi \, dx\right)^{2} + \int_{S^{1}} \phi N_{\eta}(\phi) \, dx,
  \end{equation}
  where $\phi = \eta_{x} v$ and $N_{\eta}$ is an integral operator with a smooth kernel.
\end{lemma}

\begin{proof}
  We have $\Ad_{\eta}v = \phi\circ \eta^{-1}$, where $\phi = \eta_{x}v$, so that
  \begin{equation}\label{firstestimate}
    \norm{\Ad_{\eta}v}_{\mu H^{1/2}}^{2} = 2\pi \mu^{2} + \norm{\phi\circ\eta^{-1}}^{2}_{\dot{H}^{1/2}},
  \end{equation}
  where
  \begin{equation*}
    \mu = \frac{1}{2\pi} \int_{S^{1}} \phi\circ \eta^{-1}\,dx = \int_{S^{1}} \eta_y \phi \, dy
  \end{equation*}
  under the change of variables $y=\eta(x)$, and where
  \begin{equation*}
    \norm{\zeta}^{2}_{\dot{H}^{1/2}} = \int_{S^{1}} \zeta H\zeta_{x} \, dx
  \end{equation*}
  for any $\zeta$.
  
  For the second term in \eqref{firstestimate}, the idea here is that we use the integral operator expression \eqref{hilbertintegral}, then use a change of variables to remove the composition, and finally approximate the resulting integral kernel (now a function of $\eta$) by a kernel that does not depend on $\eta$, at the expense of adding a smooth integral kernel.
  
  Specifically for $K(q) = \frac{1}{2\pi} \cot{\tfrac{q}{2}}$ we have
  \begin{align*}
    \norm{\phi\circ\eta^{-1}}^{2}_{\dot{H}^{1/2}} & = \int_{S^{1}} \int_{S^{1}} \phi\left(\eta^{-1}(w)\right) K(w-z) \frac{d}{dz} \left[ \phi\left( \eta^{-1}(z)\right)\right] \, dz\, dw \\
                                                  & = \int_{S^{1}} \int_{S^{1}} \phi(x) \eta'(x) K\left( \eta(x) - \eta(y)\right) \, \frac{d}{dy} \phi(y)\,dy\,dx,                        
  \end{align*}
  using the change of variables $w=\eta(x)$ and $z=\eta(y)$.
  Now if $\eta$ is smooth, we want to simplify $K\left(\eta(x)-\eta(y)\right)$. Recall that for $q\in[-\pi,\pi]$, the only singularity of $K(q)$ is at $q=0$, where it
  looks like $K(q) \approx \frac{1}{\pi q}$, and the difference $L(q) = \frac{1}{2\pi} \cot{\tfrac{q}{2}} - \frac{1}{\pi q}$ is easily seen to be $C^{\infty}$ on $[-\pi,\pi]$.
  Thus we can write
  \begin{equation*}
    \eta'(x) K\left( \eta(x)-\eta(y)\right) - K(x-y) = \eta'(x) L\left(\eta(x)-\eta(y)\right) - L(x-y) + \left(\frac{\eta'(x)}{\pi \left[ \eta(x)-\eta(y)\right]} - \frac{1}{\pi (x-y)}\right),
  \end{equation*}
  where the first two terms are smooth in $x$ and $y$ since $L$ and $\eta$ are. On the other hand the last term can be expanded in a series in $y$ (fixing $x$) to obtain
  \begin{equation*}
    \frac{\eta'(x)}{\eta(x)-\eta(y)} - \frac{1}{x-y} = \frac{\tfrac{1}{2}\eta''(x) + \tfrac{1}{6} \eta''(y) (y-x) + \cdots}{\eta'(x) + \tfrac{1}{2} \eta''(x)(y-x) + \cdots},
  \end{equation*}
  and since $\eta'(x)$ cannot be zero because $\eta$ is a diffeomorphism, we see that the right side of this is also smooth as a function of $x$ and $y$.
  
  We therefore have
  \begin{equation}\label{lastadad}
    \norm{\Ad_{\eta}v}^{2} = 2\pi \mu^{2} + \int_{S^{1}} \int_{S^{1}} \phi(x) K(x-y) \phi'(y) \, dy + \int_{S^{1}}\int_{S^{1}} \phi(x) \phi'(y) M_{\eta}(x,y) \, dy\,dx,
  \end{equation}
  where $M_{\eta}$ is some smooth function on $S^{1}\times S^{1}$ depending on $\eta$ and its derivatives.
  
  Let $\varepsilon>0$ be a small number; then the inequality $2ab\le \tfrac{1}{\varepsilon} a^{2} + \varepsilon b^{2}$ implies that
  the middle term in \eqref{lastadad} can be written as
  \begin{equation}\label{cauchyschwarz}
    \int_{S^{1}} \int_{S^{1}} \phi(x) K(x-y) \phi'(y) \, dx\,dy = \int_{S^{1}} \phi H\phi' \, dx
    \le \tfrac{1}{2\varepsilon} \int_{S^{1}} \phi^{2} + \tfrac{\varepsilon}{2} \int_{S^{1}} \phi'^{2} \, dx,
  \end{equation}
  using the fact that $H$ is an isometry in $L^2$.
  Finally defining $N_{\eta}(\phi)(x) = \int_{S^{1}} M_{\eta}(x,y) \phi'(y) \, dy$, the fact that $M_{\eta}$ is smooth
  implies that $N_{\eta}$ is a $C^{\infty}$-valued operator.
\end{proof}

Now using the approximation from Lemma~\ref{adadestimatelemma} to estimate the index form from Lemma~\ref{indexformlemma}, we obtain a very simple local criterion for the existence of conjugate points along a geodesic. This criterion implies that we can find conjugate points along virtually any geodesic arising from Jacobi fields that are supported in a neighborhood of any point of $S^{1}$, as happens for the 3D Euler equation in~\cite{Pre2006}. As in that paper, we conclude that the exponential map cannot be Fredholm since the first conjugate point along a geodesic is either of infinite order or a limit point of other conjugate points. In addition, we will see in Section~\ref{sec:BKM} that the criterion we derive here for conjugate points is intimately connected with the Beale-Kato-Majda criterion for blowup, similarly to what happens with the 3D Euler equation~\cite{Pre2010}.

\begin{theorem}\label{conjugateexistence}
  Let $\eta$ be a smooth geodesic in $\Diff(S^{1})$ in the $\mu H^{1/2}$-metric which is defined on the time interval $[0,T]$.
  Let $0<a<b<T$. Then there is some constant $R$ such that $\eta$ is not minimizing on $[a,b]$ whenever, for some $x_{0}\in S^{1}$, we
  have the inequality
  \begin{equation*}
    \abs{\omega_{0}(x_{0})} \int_{a}^{b} \frac{d\tau}{\eta_{x}(\tau,x_{0})^{2}} > R \pi.
  \end{equation*}
  For example, $R=4/3$ works.
\end{theorem}
The proof of this result works similarly for the full $H^{1/2}$-metric and for the degenerate $\dot{H}^{1/2}$-metric.

\begin{proof}
  By Lemma~\ref{indexformlemma}, we just need to find a test field $v$ such that $J=TL_{\eta}(v)$ gives $I(J,J)< 0$, and the idea will be to use a $v$ that is sharply peaked near a point $x_{0}$.
  
  We have already estimated $\norm{\Ad_{\eta}\dot{v}}^{2}_{\mu H^{1/2}}$ in Lemma~\ref{adadestimatelemma}; the other
  term $\llangle u_{0}, \ad_{v}\dot{v}\rrangle$ in the index form is easier: writing
  \begin{equation}
    \omega_{0}(x) = \Lambda u_{0}(x) = \frac{1}{2\pi} \int_{S^{1}}u_{0} \, dx + H\partial_{x} u_{0}(x)
  \end{equation}
  we have
  \begin{align*}
    \int_{a}^{b} \llangle u_{0}, \ad_{v}v_{t}\rrangle \,dt & = \int_{a}^{b} \int_{S^{1}} (\Lambda u_{0}) (v_{t}v_{x} - v v_{tx}) \, dx\,dt = 2\int_{a}^{b} \int_{S^{1}} \omega_{0}(x) v_{t}(t,x)v_{x}(t,x)\,dx\,dt 
  \end{align*}
  Here we used integration by parts in time and the fact that $v(a)=v(b)=0$.
  
  Using Lemma~\ref{adadestimatelemma}, the index form \eqref{indexform} is then bounded by
  \begin{equation}\label{indexformgreek}
    I(J,J) = \int_{a}^{b} \Big[ \tfrac{1}{2\varepsilon} \alpha(t) + \tfrac{\varepsilon}{2} \beta(t) + 2\gamma(t) + \mu(t) + \nu(t) \Big] \, dt,
  \end{equation}
  where
  \begin{gather*}
    \alpha(t) = \int_{S^{1}} (\eta_{x} v_{t})^{2} \, dx,        \qquad
    \beta(t) = \int_{S^{1}} \left[ \partial_{x}(\eta_{x}v_{t})\right]^{2} \, dx,  \qquad
    \gamma(t) = \int_{S^{1}} \omega_{0} v_{t} v_{x} \, dx,        \\
    \mu(t)  = \int_{S^{1}} (\eta_{x}v_{t}) N_{\eta}(\eta_{x}v_{t}) \, dx,    \qquad
    \nu(t)  = \frac{1}{2\pi} \left( \int_{S^{1}} \eta_{x}^{2} v_{t} \, dx\right)^{2}.
  \end{gather*}
  The idea is now as follows: if $v(t,x)$ is spatially supported in a small $\varepsilon$-neighborhood of $x_{0}$,
  then terms involving no $x$-derivatives of $v$ will be $O(\varepsilon)$, while terms with two $x$-derivatives
  will be $O(1/\varepsilon)$, and terms involving a single $x$-derivative will be $O(1)$. Hence in \eqref{indexformgreek}
  we will get $O(1)$ contributions from $\alpha$, $\beta$, and $\gamma$, while the terms $\mu$ and $\nu$ will only be $O(\varepsilon)$
  and can be neglected.
  
  To see how this works, we will analyze $\alpha(t)$ in more detail. Fix $t=t_{0}$ for the moment, and suppose
  $v_{t}(t_{0},x) = \zeta(\tfrac{x-x_{0}}{\varepsilon})$ for some $\varepsilon$, with $\zeta$ supported in an interval $(c,d)$.
  Then changing variables to $z=(x-x_{0})/\varepsilon$ we get
  \begin{align*}
    \alpha(t_{0}) & = \int_{S^{1}} \eta_{x}(t_{0},x)^{2} v_{t}(x)^{2} \, dx = \varepsilon \int_{c}^{d} \eta_{x}(t_{0}, x_{0} + \varepsilon z)^{2} \zeta(z)^{2} \, dz \\
                  & = \varepsilon \eta_{x}(t_{0},x_{0})^{2} \int_{c}^{d} \zeta(z)^{2} \, dz + O(\varepsilon^{2}),                                                    
  \end{align*}
  using spatial smoothness of $\eta$. Similarly we have
  \begin{equation*}
    \beta(t_{0}) = \frac{1}{\varepsilon} \eta_{x}(t_{0},x_{0})^{2} \int_{c}^{d} \zeta''(z)^{2} \, dz + O(1),
  \end{equation*}
  while $\mu(t) = O(\varepsilon)$ and $\nu(t) = O(\varepsilon^{2})$.
  The term $\gamma(t)$ requires a slightly different analysis since it involves both $v_{t}$ and $v_{x}$, so it cannot
  be analyzed just at a single time $t_{0}$, but the idea is similar.
  
  Finally we construct the test field $v$ explicitly.
  Let $f\colon [a,b]\to\RR$ and $g\colon \RR\to\RR$ be functions such that $f(a)=f(b)=0$ and $g$ is smooth with compact support in $(-m,m)$ for some $m>0$. Set $v(t,x) = f(t) g(\tfrac{x-x_{0}}{\varepsilon} - c(t))$, for some arbitrary chosen point
  $x_{0}$ and a function $c(t)$ to be chosen later. Obviously we have
  \begin{equation*}
    v_{t}(t,x) = f'(t) g\left(\frac{x-x_{0}}{\varepsilon} - c(t)\right) - c'(t) f(t) g'\left(\frac{x-x_{0}}{\varepsilon} - c(t)\right).
  \end{equation*}
  
  Now we can evaluate each term of the index form \eqref{indexformgreek} separately. We first have
  \begin{align*}
    \frac{1}{2\varepsilon} \int_{a}^{b} \alpha(t) \,dt & =                                                                                                                        
    \int_{a}^{b} \eta_{x}(t,x_{0})^{2} \int_{S^{1}} \left[ f'(t)^{2}g(z-c(t))^{2} - 2f(t)f'(t) c'(t)g(z-c(t)) g'(z-c(t)) \right. \\
                                                       & \quad \left. + f(t)^{2} c'(t)^{2} g'(z-c(t))^{2} \right] \, dx \, dt + O(\varepsilon)                                    \\
                                                       & = \frac{1}{2} \int_{a}^{b} \eta_{x}(t,x_{0})^{2} \left[ f'(t)^{2} A + c'(t)^{2} f(t)^{2} B\right] \, dt + O(\varepsilon) 
  \end{align*}
  where $A = \int_{-m}^m g(z)^{2} \, dz$ and $B = \int_{-m}^m g'(z)^{2} \, dz$, and we used the compact support of $g$
  to get $\int_{S^{1}} g(z) g'(z) \, dz = 0$.
  
  In the same way we get
  \begin{equation*}
    \frac{\varepsilon}{2} \int_{a}^{b} \beta(t) \, dt = \frac{1}{2} \int_{a}^{b} \eta_{x}(t,x_{0})^{2} \left[ f'(t)^{2} B + c'(t)^{2} f(t)^{2} C\right] \, dt + O(\varepsilon)
  \end{equation*}
  where $C = \int_{-m}^m g''(z)^{2}\, dz$, and
  \begin{equation*}
    \int_{a}^{b} \gamma(t) \, dt = - B\omega_{0}(x_{0}) \int_{a}^{b} c'(t) f(t)^{2} \, dt + O(\varepsilon).
  \end{equation*}
  Plugging into \eqref{indexformgreek} we obtain
  \begin{multline*}
    I(J,J) \le \frac{1}{2} \int_{a}^{b} \eta_{x}(t,x_{0})^{2} \Big[ Af'(t)^{2} + B c'(t)^{2} f(t)^{2} + Bf'(t)^{2} + C c'(t)^{2} f(t)^{2} \Big] \, dt \\
    - 2B \omega_{0}(x_{0}) \int_{a}^{b} c'(t) f(t)^{2} \, dt + O(\varepsilon).
  \end{multline*}
  Now define
  \begin{equation}\label{cdef}
    j(t) = \int_{0}^t \frac{d\tau}{\eta_{x}(\tau,x_{0})^{2}},
  \end{equation}
  and set $s=j(t)$ to be a rescaled time variable, and set $c(t) = kj(t)$ for some constant $k$; then we get
  \begin{equation}\label{indexformsimplest}
    I(J,J) = O(\varepsilon) + \frac{1}{2} \int_{j(a)}^{j(b)} \left[ Af'(s)^{2} + Bk^{2} f(s)^{2} + Bf'(s)^{2} + Ck^{2}f(s)^{2} - 4 k \omega_{0}(x_{0}) f(s)^{2}\right] \,ds.
  \end{equation}
  Minimizing this is now trivial; we just choose $f(s) = \sin{\left(\frac{\pi(s-j(a))}{j(b)-j(a)}\right)}$ and obtain
  \begin{equation*}
    I(J,J) = O(\varepsilon) + \frac{\pi}{4\Delta} \left( A\Delta^{2} + Bk^{2} + B\Delta^{2} + Ck^{2} - 4k\omega_{0}B\right),
  \end{equation*}
  where $\Delta = \frac{\pi}{j(b)-j(a)}$, and it remains only to choose the parameter $k$ to make this as small as possible, which is trivial: we have $k = \frac{2B\omega_{0}}{B+C}$. For
  this $k$ we get
  \begin{equation*}
    I(J,J) = O(\varepsilon) + \frac{\Delta^{2}(A+B)(B+C) - 4\omega_{0}^{2}B^{2}}{4\delta(B+C)} .
  \end{equation*}
  
  The only question remaining is how small we can make $R = \sqrt{(A+B)(B+C)}/(2B)$. It is easy to check that for
  $m=\sqrt{3}\pi/2$ the function $g(y) = \cos^{3}(x/\sqrt{3})$ has $g(m)=g'(m)=g''(m)=0$ and that
  $R = 4/3$ in this case; a slight smoothing of $g$ to make it supported in $(-m,m)$ will not substantially change
  $A$, $B$, or $C$, so we can come as close to $4/3$ as desired. It is likely that there are sharper estimates
  for this minimum, which appears for example in Mitrinović~\cite[Section 2.2.3]{Mit1970}; it is easy to see that $R\ge 1$ among functions that are supported in an interval $(-m,m)$ regardless of $m$.
\end{proof}

We want to finish this section with comments on open questions and future research directions:
\begin{itemize}
  \item The metric treated in this article presents the second example of a metric with a smooth exponential map that is not Fredholm, the only other one being the $L^{2}$-metric on $\Diff_{\mu}(M^{3})$, the volume-preserving diffeomorphism group of a three-dimensional manifold~\cite{EMP2006}. This geometric similarity further suggests that the Wunsch equation is a good one-dimensional model of the 3D Euler equation. We do not know if there are any other Euler-Arnold equations which have exponential maps with similar properties, but we suspect they would also be good models.
  \item For $s>\frac12$ the exponential map of the $H^{s}$--metric is a non-linear Fredholm map, see~\cite{MP2010}. In~\cite{BBHM2013} it is proven that $s=\frac12$ is also the critical index for positive/vanishing geodesic distance. In this section we have proven that these two geometric properties have indeed the same behavior at the critical index $s=\frac12$. It is an open question to investigate the connections between these two results. This could yield a pathway for a complete characterization of positive/vanishing geodesic distance for fractional order Sobolev metrics on diffeomorphism groups of higher dimensional manifolds.
  \item Lack of Fredholmness makes it easy to construct ``local shortcuts'' in the diffeomorphism group. Although the 3D volumorphism group has nonvanishing geodesic distance due to being a submanifold of an ambient space with positive distance~\cite{EM1970}, there is a bound for the intrinsic distance in terms of the extrinsic distance in 3D due to Shnirelman~\cite{Shn1994}, which forces finite diameter of $\Diff_{\mu}(M^{3})$, while no such result can be true in two dimensions~\cite{ER1991}. Are these properties related, and can one find a direct proof of vanishing distance or distance bounds using non-Fredholmness?
\end{itemize}

\section{Blowup}
\label{sec:blowup}

In this section we focus on a detailed analysis of blowup for equation \eqref{main}. Okamoto-Sakajo-Wunsch~\cite{OSW2008}
showed that the solution blows up at time $T$ if and only if $\int_{0}^{T} \norm{u_{x}(t,x)}_{L^{\infty}} \, dt = \infty$.
We prove that this condition can be replaced with $\int_{0}^{T} \norm{\omega(t,x)}_{L^{\infty}} \, dt =\infty$,
using the method of Beale-Kato-Majda~\cite{BKM1984}, which we use to demonstrate a link to the existence of
infinitely many conjugate pairs along a blowup solution (as studied in~\cite{Pre2010}). A similar criterion was derived
by Wunsch~\cite{Wun2011} for a slightly different member of the modified Constantin-Lax-Majda family.

Castro and C\'{o}rdoba~\cite{CC2010} demonstrated that some solutions of \eqref{main} on the real
line blow up in finite time. We extend this result to the periodic case and
demonstrate a larger class of blowup solutions by using an interesting pointwise inequality for the Hilbert transform.
Along the way we show that equation \eqref{main} takes a particularly simple form in Lagrangian coordinates,
where it looks like an equation arising in the study of the blowup for the 3D axisymmetric Euler equations of an ideal
fluid~\cite{PS2015}.

\subsection{The Beale-Kato-Majda criterion}
\label{sec:BKM}

Beale, Kato, and Majda~\cite{BKM1984} proved that the Euler equations on $\RR^{3}$,
\begin{equation}
  \frac{\partial \omega}{\partial t} + [u,\omega] = 0, \qquad \omega = \curl{u}, \quad \diver{u} = 0,
\end{equation}
with divergence-free initial condition $u(0)=u_{0}$ in $H^{q}$ for $q\ge 3$, has a solution in $H^{q}$ on $[0,T]$ if and only
if
\begin{equation}\label{BKM}
  \int_{0}^{T} \norm{\omega(t)}_{L^{\infty}} \, dt < \infty,
\end{equation}
as a consequence of the easier-to-prove criterion $\int_{0}^{T} \norm{u}_{C^{1}} \, dt < \infty$. The condition \eqref{BKM} is easier
to understand since it involves only the vorticity, which is transported by the flow $\eta$ via $\omega\left(t,\eta(t,x)\right) = D\eta(t,x) \omega_{0}(x)$.

In the present situation, Okamoto et al.~\cite{OSW2008} proved that smooth solutions of \eqref{main}
exist on $[0,T]$ if and only if $\int_{0}^{T} \norm{u}_{C^{1}}\,dt<\infty$, and our goal now is to
prove that the same criterion \eqref{BKM} works in the present situation, in terms of the momentum
$\omega = Hu_{x}$ (which should be thought of as essentially a one-dimensional version of the curl operator.)
This is not automatic since the Hilbert transform is not bounded as an operator in $L^{\infty}$ (for example
the Hilbert transform of a step function has a logarithmic singularity).
The technique is similar to that of~\cite{BKM1984}, though of course simpler in the one-dimensional
compact case.

\begin{theorem}\label{BKMHilbertthm}
  Suppose $u_{0}$ is an $H^{q}$ vector field for $q\ge 2$. Then a solution of \eqref{main} exists in $H^{q}$ on
  a time interval $[0,T]$ if and only if
  \begin{equation}\label{BKMHilbert}
    \int_{0}^{T} \norm{\omega(t)}_{L^{\infty}} < \infty.
  \end{equation}
\end{theorem}

\begin{proof}
  We first want to obtain a bound of the form
  \begin{equation}\label{basicbound}
    \norm{Hf}_{L^{\infty}} \le C\left[1+\log{\norm{f'}_{L^{2}}}\right] \left[ \norm{f}_{L^{\infty}} + 1\right]
  \end{equation} for some constant $C$ and every smooth function $f\colon S^{1}\to \RR$. For simplicity
  assume the maximum of $\abs{(Hf)(x)}$ occurs at $x=0$; then we have
  \begin{equation*}
    \norm{Hf}_{L^{\infty}} = \abs{Hf(0)} = \frac{1}{2\pi} \abs{\int_{0}^{2\pi} f(y) \cot{\frac{y}{2}} \, dy}.
  \end{equation*}
  First split the interval into $[-\rho,\rho]$ and $[\rho,2\pi-\rho]$ by periodicity, for a $\rho\in (0,1)$ to be chosen later.
  We integrate by parts near the singularity at $y=0$ to obtain
  \begin{equation*}
    \lvert Hf(0)\rvert \le \frac{1}{\pi}\abs{\log{\sin{\frac{\rho}{2}} }} \big\lvert f(\rho)-f(-\rho)\big\rvert
    - \frac{2}{\pi} \int_{-\rho}^{\rho} \abs{ f'(y)} \log{\abs{\sin{\frac{y}{2}}}} \, dy
    + \frac{1}{\pi} \int_{\rho}^{2\pi -\rho} \abs{ f(y)} \abs{ \cot{\frac{y}{2}}} \, dy.
  \end{equation*}
  The first and third terms are bounded in terms of $\norm{f}_{L^{\infty}}$, while the middle term can be bounded using
  Cauchy-Schwarz in terms of $\norm{f'}_{L^{2}}$, using
  \begin{equation*}
    \int_{-\rho}^{\rho} \log{\abs{ \sin{\frac{y}{2}}}}^{2} \, dy \le 2\int_0^{\rho} (\log{\tfrac{y}{2}})^2 \, dy = 2\rho \big[2-2\log{\tfrac{\rho}{2}}+(\log{\tfrac{\rho}{2}})^2\big] \le 8 \rho \big[\log{\rho}-1\big]^{2}
  \end{equation*}
  for $\rho\le 1$.
  We obtain
  \begin{equation*}
    \abs{Hf(0)} \le C(1-\log{\rho}) \left[ \norm{f}_{L^{\infty}} + \sqrt{\rho} \norm{f'}_{L^{2}}\right].
  \end{equation*}
  Now choose $\rho = \min\{1, \norm{f'}_{L^{2}}^{-2}\}$ to obtain \eqref{basicbound}.
  
  As shown in~\cite{OSW2008}, if $u$ is a solution of \eqref{main}, then
  \begin{equation}\label{H2estimate}
    \frac{d}{dt}\int_{S^{1}} u_{xx}^{2} \, dx = -3\int_{S^{1}} u_{x}\omega_{x}^{2} \, dx - 2\int_{S^{1}} u_{x} u_{xx}^{2}\,dx \le 5\norm{u_{x}}_{L^{\infty}} \int_{S^{1}} u_{xx}^{2} \,dx,
  \end{equation}
  so that \eqref{basicbound} with $f = \omega = Hu_{x}$
  thus yields
  \begin{equation*}
    \frac{d}{dt} \log{\norm{u_{xx}}}_{L^{2}} \le C (1+\log{\norm{u_{xx}}_{L^{2}}}) \left( 1 + \norm{\omega}_{L^{\infty}}\right),
  \end{equation*}
  and the solution of this inequality is
  \begin{equation*}
    \log{\Big(1+\log{\norm{u_{xx}}}_{L^{2}}\Big)} \le C' + C \left( t + \int_{0}^t \norm{\omega(\tau)}_{L^{\infty}}\,d\tau \right).
  \end{equation*}
  Thus a bound on $\int_{0}^{T} \norm{\omega(\tau)}_{L^{\infty}} \, d\tau$ implies a double-exponential bound on the Sobolev norm
  $\norm{u(t)}_{H^{2}}$.
  
  Similar estimates as in \eqref{H2estimate} are easy to derive for higher Sobolev norms of $u$, with only $\norm{u_{x}}_{L^{\infty}}$
  being relevant, and thus we get bounds on the growth of $\norm{u(t)}_{H^{q}}$ for all $q\ge 2$ in terms solely of
  $\int_{0}^t \norm{\omega(t)}_{L^{\infty}}$.
\end{proof}

\begin{remark}\label{shrinking}
  Note that the first part of equation \eqref{H2estimate} implies that growth of the norm happens only when $u_{x}<0$, so in fact to bound the $H^{2}$-norm of $u$, it is sufficient to obtain a bound on $(u_{x})_-$, the negative part of $u_{x}$. Since particle trajectories satisfy the flow equation $\eta_{t}(t,x) = u\left(t,\eta(t,x)\right)$, we see that $u_{x}\left(t,\eta(t,x)\right) = \frac{\partial}{\partial t} \log{\eta_{x}(t,x)}$; thus in the Lagrangian analysis of blowup in the next section, we need only be concerned with $\eta_{x}$ approaching zero (rather than infinity). An alternate way to understand this is using the conservation law \eqref{conservation-law}, in the form
  \begin{equation*}
    \omega\left(t,\eta(t,x)\right) = \omega_{0}(x)/\eta_{x}(t,x)^{2},
  \end{equation*}
  to see that if $\eta_{x}(t,x)\ge a(t)$ for all $x$, then
  \begin{equation*}
    \norm{\omega(t)}_{L^{\infty}} \le \norm{\omega_{0}}_{L^{\infty}}/a(t)^{2},
  \end{equation*}
  so the blowup condition is that
  $\int_{0}^{T} dt/a(t)^{2} = \infty$; no upper bound on $\eta_{x}$ is needed.
  
  In three-dimensional fluids the stretching map $D\eta$ always has determinant one, so although here stretching is more important,
  stretching must always be accompanied by shrinking in another direction.
\end{remark}

\subsection{Proof of blowup}

In this section we demonstrate finite-time blowup of solutions to \eqref{main}; we imitate
the technique of Castro and C\'{o}rdoba~\cite{CC2010}, although the situation simplifies here
due to periodicity since we are able to work in terms of Fourier coefficients. In addition
we are able to generalize the method to deal with any initial data such that $u_{0}'(0)<0$ and
$(Hu_{0})'(0)=0$, while previously this was only known to work if $u_{0}$ was an odd function. In
fact the inequality $H(fHf'')+ff''\le 0$ we require is fairly easy to generalize to higher
derivatives, as in Theorem~\ref{signlemma}, which we hope will be useful in other applications.

To prove the result, we need to estimate a nonlocal term. Fortunately properties of the Hilbert transform
make this relatively elementary, especially using the Fourier series. Castro and C\'{o}rdoba~\cite{CC2010}
proved that if $f$ is an odd function on the line, then $H(fHf'')(0)<0$, using Mellin transforms. In our
case the proof uses Fourier coefficients, and we are able to get a stronger result, which also generalizes
a result due to C\'{o}rdoba-C\'{o}rdoba~\cite{CC2003}.

\begin{theorem}\label{signlemma}
  Suppose $f\colon S^{1}\to \RR$ is a function with Fourier series $f(x) = \sum_{n\in\mathbb{Z}} c_{n} e^{inx}$.
  Let $H$ denote the Hilbert transform, and set $\Lambda = H\partial_{x}$,
  so that $\Lambda(e^{inx}) = \abs{n} e^{inx}$ for every $n\in\mathbb{Z}$.
  For any positive number $p$, define $g_{p} = H(fH\Lambda^{p}f) + f\Lambda^{p}f$.
  Then for every $x\in S^{1}$ we have
  \begin{equation}\label{magicinequality}
    g_{p}(x) = 2\sum_{k=1}^{\infty} \left[ k^{p} - (k-1)^{p}\right] \abs{\phi_k(x)}^{2}, \qquad \text{where $\phi_k(x) = \sum_{m=k}^{\infty} c_m e^{imx}$.}
  \end{equation}
  In particular every $g_{p}$ is nonnegative and strictly positive if $f$ is nonconstant.
\end{theorem}

\begin{proof}
  It is sufficient to prove \eqref{magicinequality} at $x=0$, since if $x$ were not zero we could just replace $c_{n}e^{inx}$
  everywhere with $\tilde{c}_{n}$. This corresponds to translation-invariance of the operators $H$ and $\Lambda$, and simplifies notation.
  
  We first compute
  \begin{equation*}
    (fH\Lambda^{p}f)(x) = \sum_{m,n\in \mathbb{Z}^{2}} c_m (-i\sgn{(n)})\abs{n}^{p} c_{n} e^{i(m+n)x},
  \end{equation*}
  then observe that
  \begin{equation*}
    H(fH\Lambda^{p}f)(0) = -\sum_{m,n\in \mathbb{Z}^{2}} \sgn{(n)} \sgn{(m+n)} \abs{n}^{p} c_m c_{n}.
  \end{equation*}
  Therefore we have
  \begin{equation*}
    g_{p}(0) = \sum_{m,n\in\mathbb{Z}^{2}} \left[1 - \sgn{(n)} \sgn{(m+n)}\right] \abs{n}^{p} c_m c_{n}.
  \end{equation*}
  Of course, when $n=0$ there is no contribution since $p>0$, so we can break up the sum into
  terms when $n>0$ and when $n<0$, obtaining
  \begin{equation}\label{firstbreak}
    g_{p}(0) = \sum_{n=1}^{\infty} n^{p} c_{n} \sum_{m\in\mathbb{Z}} \left[ 1- \sgn{(m+n)}\right] c_m + \sum_{n=1}^{\infty} n^{p} \overline{c_{n}} \sum_{m\in \mathbb{Z}} \left[ 1 + \sgn{(m-n)}\right] c_m = (I)+(II),
  \end{equation}
  where in the second sum we replaced $n$ with $-n$ and used the fact that $c_{-n} = \overline{c_{n}}$ since $f$ is real-valued.
  
  Now in the sum $(I)$ of \eqref{firstbreak} we have nonzero terms if and only if $m+n\le 0$, and thus we have
  \begin{equation*}
    (I) = \sum_{n=1}^{\infty} n^{p} c_{n} \left( c_{-n} + 2\sum_{m=-\infty}^{-n-1} c_m\right)  = \sum_{n=1}^{\infty} n^{p} \abs{c_{n}}^{2} + 2 \sum_{n=1}^{\infty} \sum_{m=n+1}^{\infty} n^{p} c_{n} \overline{c_m},
  \end{equation*}
  where in the last sum we replaced $m$ with $-m$. The same tricks applied to the second term $(II)$ of \eqref{firstbreak} give
  \begin{equation*}
    (II) = \sum_{n=1}^{\infty} n^{p} \abs{c_{n}}^{2} + 2 \sum_{n=1}^{\infty} \sum_{m=n+1}^{\infty} n^{p} \overline{c_{n}} c_m.
  \end{equation*}
  Plugging these both into \eqref{firstbreak} we get
  \begin{equation}\label{secondbreak}
    g_{p}(0) = 2\sum_{n=1}^{\infty} n^{p} \abs{c_{n}}^{2} + 2\sum_{n=1}^{\infty} \sum_{m=n+1}^{\infty} n^{p} (c_{n} \overline{c_m} + c_m \overline{c_{n}}).
  \end{equation}
  
  In terms of $\phi_k = \sum_{j=k}^{\infty} c_j$, equation \eqref{secondbreak} becomes
  \begin{align*}
    g_{p}(0) & = 2\sum_{n=1}^{\infty} n^{p} \abs{\phi_{n} - \phi_{n+1}}^{2} + 2\sum_{n=1}^{\infty} n^{p} \left[ (\phi_{n}-\phi_{n+1}) \overline{\phi_{n+1}} + \phi_{n+1} (\overline{\phi_{n}} - \overline{\phi_{n+1}})\right] \\
             & = 2\sum_{n=1}^{\infty} n^{p} \left[ \abs{\phi_{n}}^{2} - \abs{\phi_{n+1}}^{2}\right],                                                                                                                          
  \end{align*}
  which becomes \eqref{magicinequality} after summation by parts.
\end{proof}

In fact the only features we used of the power function $Q(\lambda) = \lambda^{p}$ are that $Q(\lambda)=0$
and $Q$ is increasing on $\RR^+$. Hence the same proof gives the inequality
\begin{equation*}
  H(fHg) + fg \ge 0 \qquad \text{if $g = Q(\Lambda)(f)$}
\end{equation*}
for any such $Q$, where the action of $Q$ on a function is determined by linearity and the action on the Fourier basis:
$Q(\Lambda)(e^{inx}) = Q(\abs{n}) e^{inx}$. Note also that by the general product formula
$(Hf)(Hg) - H(gHf) = H(fHg)+fg,$ we obtain for free the inequality
\begin{equation*}
  (Hf)(Hg)-H(gHf) \ge 0 \qquad \text{if $g=Q(\Lambda)(f)$.}
\end{equation*}

The proof above also works (and is in fact simpler) for functions defined on $\RR$ rather than on $S^{1}$,
using Fourier transforms rather than Fourier series, as we will show in Appendix~\ref{appendix:real_line}. We will only need the
special case $p=2$ in what follows, but we summarize the result for all integers $p$ below; note that since $H$
has period $4$ we essentially get four cases.

\begin{corollary}\label{signlemmacorollary}
  For any function $f\colon S^{1}\to\RR$ we have the following pointwise inequalities:
  \begin{alignat*}{3}
    -H(ff') + fHf'    & \ge 0, & \qquad & -H(fHf'') - ff''                  & \ge 0, \\
    H(ff''') - fHf''' & \ge 0, & \qquad & H(fHf^{\imath v}) + ff^{\imath v} & \ge 0. 
  \end{alignat*}
  The same inequalities are valid when the order of the derivatives are replaced by any positive
  integer which is equal modulo $4$.
\end{corollary}

\begin{proof}
  We just use $H^{2} = -1$ and break up $\Lambda^{p} = H^{p} \partial_{x}^{p}$ depending on $p$ modulo $4$.
\end{proof}

The first inequality $-H(ff') + fHf' \ge 0$ is equivalent to the case $\alpha=1$ of the pointwise inequality
in~\cite{CC2003}, which takes the form $\Lambda(f^{2}) \le 2f\Lambda f$ in terms of $\Lambda = H\partial_{x}$.
The fact that this inequality can be generalized to any $\alpha \in [0,2]$ and to any convex function, as
well as to higher dimensions (see C\'{o}rdoba-Mart\'inez~\cite{CM2015}) suggests that Theorem~\ref{signlemma}
might be generalized to higher dimensions as well, for example replacing the Hilbert transform by Riesz
transforms. But we have not been able to generalize the result using the present techniques.

Corollary~\ref{signlemmacorollary} and the special structure of \eqref{main} (as distinct from all other equations
in the ``modified Constantin-Lax-Majda'' family) allow us to obtain the following especially simple form
in Lagrangian flow coordinates.

\begin{theorem}\label{lagrangianform}
  Suppose $u$ and $\omega$ form a solution of \eqref{main} with $\int_{0}^{2\pi} u_{0}(x) \, dx = 0$. Let
  $\eta$ denote the Lagrangian flow of $u$ satisfying
  \begin{equation}\label{lagrangianflow}
    \eta_{t}(t,x) = u\left(t, \eta(t,x)\right), \qquad \eta(0,x) = x.
  \end{equation}
  Then $\eta_{x}$ satisfies the equation
  \begin{equation}\label{lagrangiangeodesic}
    \eta_{xtt}(t,x) = \frac{\omega_{0}(x)^{2}}{\eta_{x}(t,x)^{3}} - F\left(t,\eta(t,x)\right) \eta_{x}(t,x),
  \end{equation}
  where $F(t,x) = -uu'' - H(uHu'')$ is positive for all $t$ and $x$ by Corollary~\ref{signlemmacorollary}.
\end{theorem}

\begin{proof}
  Since $\omega = Hu_{x}$, the Hilbert transform of \eqref{main} is
  \begin{equation*}
    u_{tx} = H(u\omega_{x}) + 2 H(u_{x} Hu_{x}).
  \end{equation*}
  Using the Hilbert transform identity $2H(fHf) = (Hf)^{2} - f^{2}$ (valid as long as $f$ has vanishing mean value),
  this equation becomes
  $
  u_{tx} = H(uu_{xx}) + \omega^{2} - u_{x}^{2},
  $ which may be written in the form
  \begin{equation}\label{velocityderivative}
    u_{tx} + u_{x}^{2} + uu_{xx} = \omega^{2} - F.
  \end{equation}
  Now differentiate the flow equation in space to get
  \begin{equation}\label{spacederivativeflow}
    \eta_{tx}(t,x) = u_{x}\left(t,\eta(t,x)\right) \eta_{x}(t,x),
  \end{equation}
  and differentiating this in time gives
  \begin{equation}\label{flow3derivative}
    \begin{split}
      \eta_{ttx}(t,x) = \left(u_{tx} + u_{x}^{2} + u u_{xx}\right)\left(t,\eta(t,x)\right) \eta_{x}(t,x).
    \end{split}
  \end{equation}
  Composing \eqref{velocityderivative} with $\eta$, we therefore see that it can be written in the form
  \begin{equation}\label{flow3almostdone}
    \eta_{ttx}(t,x) = \omega\left(t,\eta(t,x)\right)^2 \eta_{x}(t,x) - F\left(t,\eta(t,x)\right) \eta_{x}(t,x).
  \end{equation}
  
  Plugging the vorticity conservation law
  \eqref{conservation-law} into \eqref{flow3almostdone} gives \eqref{lagrangiangeodesic}.
  The fact that $F$ is always positive comes from the fact that $u$ can never be constant, since
  $\int_{0}^{2\pi} u(t,x)\,dx = 0$ for all $t$; if $u$ were constant at any time, then $u$ would have
  to be identically zero at that time, which would mean $u$ is identically zero for all time. Since
  $u$ is nonconstant, $F$ is strictly positive.
\end{proof}

Our first blowup result is now quite easy in this Lagrangian form.

\begin{theorem}\label{blowuptheorem}
  Suppose $u$ and $\omega$ form a solution of \eqref{main} where $u_{0}$ and $\omega_{0} = H\partial_{x} u_{0}$ satisfy
  $\omega_{0}(x_{0})=0$ and $u_{0}'(x_{0})<0$ for some $x_{0}\in S^{1}$, with $\int_{0}^{2\pi} u_{0}(x)\,dx=0$. Let $\eta$
  be the Lagrangian flow of $u$.
  Then $\eta_{x}(t,x_{0})$ reaches zero in finite time $T < 1/\abs{u_{0}'(x_{0})}$, and thus
  $u_{x}\left(t,\eta(t,x_{0})\right)$ reaches negative infinity at the same time.
\end{theorem}

\begin{proof}
  Since $\omega_{0}(x_{0})=0$, the conservation law \eqref{conservation-law} implies that $\omega\left(t,\eta(t,x_{0})\right) = 0$
  for all time. Equation \eqref{lagrangiangeodesic} thus implies that $\eta_{ttx}(t,x_{0}) < 0$ as long as $\eta_{x}(t,x_{0})>0$.
  Hence the function $\phi(t)=\eta_{x}(t,x_{0})$ has $\phi(0)=1$, $\phi'(0)=u_{x}(0,x_{0})<0$, and $\phi''(t)<0$, so that $\phi(t)$ must reach
  zero in finite time. Since equation \eqref{spacederivativeflow} implies
  \begin{equation*}
    \eta_{x}(t,x_{0}) = \exp{\left( \int_{0}^t u_{x}\left(\tau,\eta(\tau,x_{0})\right) \, d\tau\right)},
  \end{equation*}
  we must have $u_{x}\left(t,\eta(t,x_{0})\right)$ approaching negative infinity as well.
\end{proof}

Equation \eqref{lagrangiangeodesic} is exactly the Ermakov-Pinney equation~\cite{LA2008} for $r(t) = \eta_{x}(t,x)$ with
$\Omega(t)^{2} = F(t,\eta(t,x))$ and $c = \omega_{0}(x)$; that is,
\begin{equation}\label{ermakovpinney}
  r''(t) + \Omega(t)^{2} r(t) = \frac{c^{2}}{r(t)^{3}}.
\end{equation}
As pointed out by Eliezer and Gray~\cite{EG1976}, this equation appears when describing the motion of a time-dependent
harmonic oscillator in the plane, with a returning force always pointing
toward the origin. If we consider the system $\ddot{\mathbf{r}}(t) + \Omega(t)^{2} \mathbf{r}(t) = 0$
for a vector $\mathbf{r}(t)\in \RR^{2}$, then $r(t) = \norm{\mathbf{r}(t)}$ solves
\eqref{ermakovpinney}, where $c$ is the (constant) angular momentum.

Since $\eta(0,x)=x$ for all $x$, we always have initial condition $r(0)=1$; that is, the
``particle'' begins on the unit circle with some initial radial velocity (which could be positive
or negative) and some initial angular momentum (which is preserved). If the angular momentum
is zero, then the particle remains on a line between the origin and its starting position on
the circle. If it is initially heading inward (as in the situation of Theorem~\ref{blowuptheorem})
then it reaches the origin in finite time. If it heads outward but not too fast, it is likely that
the (nonlocal) returning force may still send it to the origin, but we cannot tell yet without sharp
bounds on $F = -uu'' - H(uHu'')$ as given by Lemma~\ref{signlemma}. Note that in this case it is not
necessary that $\Omega(t)$ actually approaches infinity in some finite time, only that it remains positive,
for blowup to occur. If $\Omega(t)$ remains finite, then when $r(t)$ approaches zero we must have $r'(t)$
approaching some nonzero value, which means $r(t) \approx C(T-t)$ for some $C\ne 0$ as $t\to T$. We thus
have $\int_{0}^{T} \frac{dt}{\eta_{x}(t,x_{0})^{2}} = \infty$ as we might expect from Theorem~\ref{blowuptheorem}.

At present we can only conjecture that blowup can be localized, in the sense that a solution should blow up
if and only if there is some $x_{0}\in S^{1}$ such that
\begin{equation}\label{localizedblowup}
  \int_{0}^{T} \abs{\omega\left(t,\eta(t,x_{0})\right)} \, dt = \abs{\omega_{0}(x_{0})} \int_{0}^{T} \frac{dt}{\eta_{x}(t,x_{0})^{2}} = \infty.
\end{equation}
(We do not need the maximum of $\omega$ to actually be reached along a particular Lagrangian path, only that
some path carries enough of the vorticity stretching to still give infinity here.) Note that the scenario
of Theorem~\ref{blowuptheorem} does not quite work here since $\omega_{0}(x_{0}) = 0$; however one would expect
that if $\omega_{0}'(x_{0})\ne 0$, then any point near $x_{0}$ would still give infinity.

On the other hand, if there is any angular momentum, then the particle spins around the origin, and the
localized blowup criterion \eqref{localizedblowup} is exactly the condition that $c \int_{0}^{T} dt/r(t)^{2} = \infty$.
Since in polar coordinates we have $\theta'(t) = c/r(t)^{2}$ by conservation of angular momentum, we can have
blowup along such a trajectory at time $T$ if and only if the corresponding planar system has a solution that completes
infinitely many turns around the origin before time $T$. Such a system necessarily requires the force $\Omega(t)$
to be approaching infinity as $t\to T$ already.

This analysis is \emph{exactly} what happens for the 3D axisymmetric Euler equations with swirl when considering the flow
near the axis of symmetry (and perhaps more generally), as shown by Sarria and the third author~\cite{PS2015}.
There the vorticity conservation law also translates into a constant angular velocity in the Ermakov-Pinney equation,
and the central force corresponds to the second radial derivative of the pressure. It is conjectured but not yet
known if the pressure has a local minimum on the axis of symmetry, but if it did this would correspond to the
``magical'' inequality of Lemma~\ref{signlemma} that makes the force always point inwards. This fact is what
makes us convinced that \eqref{main} is the best one-dimensional model of the 3D Euler equation.

Finally we note here that the localized blowup criterion \eqref{localizedblowup}
is exactly the condition needed in Theorem~\ref{conjugateexistence} to obtain an infinite sequence
$(0,t_{1},t_2,\ldots)$ with $t_k\nearrow T$ such that $\eta(t_k)$ is monoconjugate to $\eta(t_{k+1})$ for
every $k$; in other words the geodesic is not locally minimizing even on short time intervals. (Note
that in spite of the vanishing geodesic distance, geodesics are still locally minimizing in the $H^{2}$
topology due to smoothness of the exponential map; they just fail to be globally minimizing.) This
is a slightly simpler version of the phenomenon described in~\cite{Pre2010} for 3D fluid flow, again
illustrating how similar \eqref{main} is to the 3D Euler equation when one looks at them both geometrically.

Hence we expect a complete understanding of blowup in the Wunsch equation to lead to a better understanding
of possible blowup scenarios for the 3D Euler equation. In particular the following open questions should be
addressed:
\begin{itemize}
  \item We know that blowup will occur if $\omega_{0}(x_{0})=0$ and $u_{0}'(x_{0})<0$ for some $x_{0}\in S^{1}$. Does
        it in fact occur if $u_{0}'(x_{0})\ge 0$? If so, it would imply that every nonconstant solution ends in finite
        time, since $\int_{S^{1}} \omega_{0} \, dx = 0$ implies $\omega_{0}$ vanishes at one point.
  \item The work of Khesin-Lenells-Misio{\l}ek~\cite{KLM2008} implies that a nonzero mean term $\mu$ can prevent
        blowup in the Hunter-Saxton equation. So far all of our analysis has assumed that the mean of the initial data
        (and hence of all future data) is zero; can blowup still occur if it is not?
  \item Theorem~\ref{BKMHilbertthm} implies that at the blowup time we have
        $\int_{0}^{T} \norm{\omega(t)}_{L^{\infty}} \, dt = \infty$. On the other hand it appears that blowup
        is localized to a point $x_{0}$ where $\omega_{0}(x_{0})=0$, and hence we do not have
        $\int_{0}^{T} \abs{\omega\left(t,\eta(t,x_{0})\right)} \, dt = \infty$ (the localized Beale-Kato-Majda criterion)
        since momentum is conserved. Is this localized criterion valid for $x$ in a selected neighborhood of $x_{0}$?
  \item The localized Beale-Kato-Majda criterion is precisely what shows up in the condition for an infinite
        sequence of conjugate pairs leading up to the blowup time as in Theorem~\ref{conjugateexistence}. Are there
        such conjugate pairs, as conjectured in~\cite{Pre2010}? Can one see the failure of minimization along the
        corresponding geodesic explicitly?
  \item The Camassa-Holm equation (the Euler-Arnold equation for the $H^{1}$-metric on $\Diff(S^{1})$)
        has solutions which blow up in finite time in the sense that $u_{x}$ approaches
        negative infinity; however one can define global weak solutions, thus ensuring the continuation of the Lagrangian flow beyond the blow-up point;
        see ~\cite{McK2003,BC2007,BCZ2015}. Does the same thing happen for the Wunsch equation? That is, can one extend geodesics in $C^{\infty}(S^{1},S^{1})$
        for all time even if they leave $\Diff(S^{1})$?
  \item It is easy to see that the positive forcing term $F = -uu_{xx}-H(uHu_{xx})$ given by Theorem~\ref{lagrangianform}
        can be bounded above in terms of $\norm{u}_{H^{3/2}}^{2}$. On the other hand it is not clear how
        to obtain a good lower bound, although we know it is positive. A lower bound is necessary to prove blowup
        of other Lagrangian trajectories carrying nonzero momentum.
  \item Finally, it should be remembered that for $s >3/2$, the $H^{s}$-metric and the $\mu H^{s}$-metric on $\Diff(S^{1})$ are geodesically complete (geodesics are defined for all time)~\cite{EK2014a}. The critical exponent $s=3/2$ was studied in~\cite{GBR2015}. It was shown there that the metric is strong and geodesically complete on a manifold which could be thought as a replacement for $\D{3/2}$ (which is not a topological group). However, it is not clear if the metric is geodesically complete on $\Diff^{\infty}(S^{1})$ (local well-posedness in that case was studied in~\cite{EK2014}). It would also be interesting to study the blowup for other values of $s\le 3/2$.
\end{itemize}

\section{Curvature}\label{sec:curvature}

In this part we will show that the sectional curvature for
the $\mu H^{1/2}$--metric admits both signs and is locally unbounded from above.
The unboundedness of the curvature is further evidence that there exists a relation between unbounded sectional curvature and vanishing geodesic distance as conjectured by Michor and Mumford~\cite{MM2005}.

Formulas for the sectional curvature of $\Diff(S^1)/\Rot(S^1)$ were computed by Bowick-Rajeev~\cite{BR1987} and Zumino~\cite{Zum1988} using the formula of Freed~\cite{Fre1988} for K\"{a}hler geometry, by Kirillov-Yur'ev~\cite{KY1987} using complex normal coordinates, and by Gordina-Lescot~\cite{GL2006} directly using the covariant derivative. Here we will use Arnold's curvature formula~\cite{Arn1966} for right invariant metrics on Lie-groups, which allows us to see the effect of the mean term $\mu$.
The sectional curvature $K(u,v)$ for a right-invariant metric $\langle\cdot,\cdot\rangle$ on a Lie group $G$ at the unit element $e$ for the $2$-plane spanned by linearly independent vectors $u,v \in \mathfrak g$ is given by
\begin{equation}\label{eq:curvature}
  K(u,v) =\frac{\langle\mathcal R(u,v)v,u\rangle}{\langle u,u\rangle\langle v,v\rangle-\langle u,v\rangle^{2}}= \frac{\langle\delta,\delta\rangle - 2\langle\alpha,\beta\rangle - 3\langle\alpha,\alpha\rangle - 4\langle B_{u},B_{v}\rangle}{\langle u,u\rangle\langle v,v\rangle-\langle u,v\rangle^{2}}
\end{equation}
where
\begin{equation*}
  2\alpha = \ad_uv, \qquad 2\beta = \ad_{u}^{\top}v - \ad_{v}^{\top}u, \qquad 2\delta = \ad_{u}^{\top}v + \ad_{v}^{\top}u
\end{equation*}
and
\begin{equation*}
  2B_{u} = \ad_{u}^{\top}u, \qquad 2B_{v} = \ad_{v}^{\top}v.
\end{equation*}

It remains to calculate the terms of the curvature formula for
the $\mu H^{1/2}$-metric on $\Diff(S^{1})$. Therefore we recall
the definition of the inertia operator and the adjoint operator:
\begin{equation}
  \Lambda u = \mu(u) + Hu_{x},\qquad \ad_{u}^{\top}v = \Lambda^{-1}\left(2u_{x}\Lambda(v) + u\Lambda(v_{x})\right).
\end{equation}
Plugging this into Arnold's curvature formula yields the general formula for the sectional curvature of the
$\mu H^{1/2}$-metric. However, since the resulting formula does not seem to give any insights we will refrain from doing so and instead
calculate the sectional curvature for the particular choice of vector fields $u,v\in\{ \sin(mx),\cos(nx)\}$. We obtain the following result:

\begin{theorem}\label{thm:curv}
  The sectional curvature of the $\mu H^{1/2}$-metric admits both signs and is locally unbounded from above.
  In particular we have:
  \begin{align}
    K(\sin(mx),\sin(nx)) & =K(\sin(mx),\cos(nx))=K(\cos(mx),\cos(nx))                          \\
                         & =\frac{m \left(m^{2}+2 m n+2 n^{2}\right)}{2 n (m+n)},\qquad n>m>0  \\
    K(\sin(mx),\cos(nx)) & =\frac{n \left(2 m^{2}+2 m n+n^{2}\right)}{2 m (m+n)}, \qquad m>n>0 \\
    K(\sin(mx),\cos(mx)) & =\frac{1}{2}(5m-6), \qquad m>0\,.                                   
  \end{align}
\end{theorem}

\begin{proof}
  Let $n,m \ne 0$. Then we have
  \begin{equation*}
    \Lambda(\sin(nx)) = \abs{n} \sin(nx), \qquad \Lambda(\cos(mx)) = \abs{m} \cos(mx)
  \end{equation*}
  and integrating we get
  \begin{equation*}
    \norm{\sin(nx)}_{\mu H^{1/2}}^{2} = \pi \abs{n}, \quad \norm{\cos(mx)}_{\mu H^{1/2}}^{2} = \pi \abs{m}.
  \end{equation*}
  We will also need that $\sin(nx)$ and $\cos(mx)$ are orthogonal with respect to the $\mu H^{1/2}$-metric, i.e.,
  \begin{equation*}
    \llangle\sin(nx), \cos(mx)\rrangle_{\mu H^{1/2}} = \llangle \sin(nx), \sin(mx)\rrangle_{\mu H^{1/2}} = \llangle\cos(nx), \cos(mx)\rrangle_{\mu H^{1/2}} = 0.
  \end{equation*}
  We will only present the remaining calculations for the case $u=\sin(mx)$ and $v=\sin(nx)$ with $n>m>0$.
  For the operator $\ad^{\top}_u v$ we calculate:
  \begin{align*}
    \ad^{\top}_u v & =n \Lambda^{-1}\left( 2m\cos(mx) \sin(nx) +n\sin(mx) \sin(nx)\right)        \\
                   & =n\frac{m-\frac{n}2}{n-m} \sin((m-n)x)+n\frac{m+\frac{n}2}{m+n}\sin((m+n)x) 
  \end{align*}
  To calculate $\Lambda^{-1}$ we used that $H^{-1}=-H$ and that $\sin(k x)$ has zero mean for any $k\in \mathbb N$.
  Thus we obtain the formulas for $\beta, \delta, B_u$ and $B_{v}$:
  \begin{align*}
    2\beta  & = \frac{n^{2}-m^{2}}{2(n+m)}\sin((m+n)x)- \frac{m^{2}-4mn+n^{2}}{2(n-m)} \sin((n-m)x), \\
    2\delta & =\frac{n^{2}-m^{2}}{2(n-m)}\sin((n-m)x)+ \frac{m^{2}+4mn+n^{2}}{2(m+n)} \sin((m+n)x),  \\
    B_u     & = \frac34 m \sin(2mx), \qquad B_v = \frac34 n\sin{2nx}.                                
  \end{align*}
  For the remaining term $\alpha$ we have
  \begin{align}
    \alpha & = -\frac{n-m}2\sin((m+n)x)+\frac{n+m}2\sin((n-m)x), 
  \end{align}
  Now we can calculate the $\mu H^{1/2}$-norms. Since $\sin(kx)$ is an orthogonal system with respect to the $\mu H^{1/2}$-metric we have
  \begin{align*}
    \llangle B_u, B_{v}\rrangle_{\mu H^{1/2}} =0\,. 
  \end{align*}
  For the other terms we arrive at:
  \begin{align*}
    4\llangle \delta,\delta\rrangle & = \frac{\pi}{2(n+m)}\left(n (3 m^{3} + 9 m^{2} n + 5 m n^{2} + n^{3}) \right) \\
    2\llangle \alpha,\beta \rrangle & = -\frac{\pi}{8}(m + n) (m^{2} - 3 m n + n^{2})                               \\
    \llangle \alpha,\alpha \rrangle & = \frac{\pi}{8} (n - m) n (m + n)                                             
  \end{align*}
  Putting everything together we obtain for the sectional curvature:
  \begin{align*}
    K(\sin(mx),\sin(nx))
      & =\frac{\pi m (m^{2} + 2 m n + 2 n^{2})}{4 (m + n) n}\geq \frac{\pi m (m+n)}{4 n} 
  \end{align*}
  which is unbounded for fixed $n$ and $m\rightarrow \infty$.
  Note, that this formula is not symmetric in $u$ and $v$, since we have assumed
  that $n>m$, in order to calculate the Hilbert transform of terms
  as $\sin((n-m)x)$. The calculations for the other terms are similar. Only in the case $u=\cos(mx)$ $v=\sin(mx)$, one has to be more careful since $\Lambda(\ad^\top_uv)$ has non-zero mean, and thus one has to take into account the $\mu$-term for the inversion of $\Lambda$.
\end{proof}

Several questions concerning the sectional curvature remain open for future investigation:
\begin{itemize}
  \item
        We have seen that the sectional curvature of the $\mu H^{1/2}$-metric has unbounded positive curvature. This has also been observed for the $L^{2}$-metric.
        Both of these metrics have vanishing geodesic distance. On the other hand, for metrics of order one --- which induce non-vanishing geodesic distance --- the curvature turns out to be bounded from above.
        It is an open conjecture by Michor and Mumford that there is a relation between unbounded curvature and vanishing geodesic distance.
  \item
        We have seen that the sectional curvature admits both signs and is unbounded from above. It is not known, whether the curvature is bounded from below. In this case it would be interesting to derive explicit lower bounds.
  \item
        The homogeneous $H^{1}$--metric has strictly positive sectional curvature. On the other hand, it has been shown in~\cite{KLM2008} that the sectional curvature of the $\mu H^{1}$--metric admits both signs. Similarly, the negative sectional curvature directions, that are described in Thm.~\ref{thm:curv}, are originating only from the $\mu$-term in the $\mu H^{1/2}$--metric. It remains open whether the sectional curvature of the homogeneous $\dot{H}^{1/2}$-metric admits both signs or is strictly positive.
  \item
        One can repeat the above calculations for the general homogeneous $\dot{H}^{s}$-metric.
        For $0<m<n$ we find that
        \begin{align*}
            & K(\sin(mx),\sin(nx))=\frac{\pi}{4} (nm)^{-2s} \bigg((-m+n)^{-2s} \left(m^{1+2s}-2 m^{2s} n+2 m n^{2s}-n^{1+{2s}}\right)^{2} 
          \\&\qquad+(m+n)^{-{2s}} \left(m^{1+{2s}}+2 m^{2s} n+2 m n^{2s}+n^{1+{2s}}\right)^{2}-4 m^{2s}\left(-m^{2}+2 n^{2}\right)\\&\qquad+4 n^{2s}\left(-2 m^{2}+n^{2}\right)-3 (-m+n)^{2} (m+n)^{2s}-3 (m+n)^{2} (n-m)^{2s}\bigg)\,.
        \end{align*}
        For $s=0$ this gives
        \begin{align*}
            & K(\sin(mx),\sin(nx))=3\pi \left( m^{2}+n^{2} \right)\,, 
        \end{align*}
        which is the sectional curvature of the metric associated to Burgers equation.
        For $s=1$ this gives constant sectional curvature, as computed by Lenells~\cite{Len2007} for the homogeneous $\dot{H}^{1}$-metric, which is associated to the Hunter Saxton equation. Clearly the sectional curvature for $s=0$ is not bounded, while it is bounded for for $s=1$. It remains open to study boundedness of the sectional curvature for general $H^{s}$--metrics.
\end{itemize}

\appendix

\section{The homogeneous $\dot{H}^{1/2}$-metric}
\label{appendix:homogeneous-metric}

The $\dot{H}^{1/2}$ inner product on $\CS$ is defined as
\begin{equation}
  \llangle u, u\rrangle_{\dot{H}^{1/2}} = \int_{S^{1}} u Hu_{x} \, dx = 2\pi \sum_{n\in \mathbb{Z}} \abs{n} \abs{\hat{u}_{n}}^{2} \quad \text{if}\quad u(x) = \sum_{n\in\mathbb{Z}} \hat{u}_{n} e^{inx}.
\end{equation}
This inner product does not induce a right-invariant metric on $\Diff(S^{1})$ because it is \emph{degenerate}. The latter can be avoided if we work on $\Diff(S^{1})/\Rot(S^{1})$, corresponding to the homogeneous space of diffeomorphisms modulo rotations.

More generally, the geodesic equation for a right-invariant Riemannian metric on a homogeneous space $G/K$ can be reduced to the \emph{Euler-Poincar\'{e}} equation
\begin{equation}\label{eq:Euler-Poincare}
  \omega_{t} = \ad^{*}_{u}\,\omega, \qquad \omega \in \mathfrak{g}^{*},
\end{equation}
on the dual space $\mathfrak{g}^{*}$ of the Lie algebra of $G$ (see Poincar\'{e}'s original paper~\cite{Poi1901}).

But unfortunately, in general, there is no well-defined \emph{Euler-Arnold equation} in that case because the Eulerian velocity (defined using a lift $g(t)$ in $G$ of a path $x(t)$ in $G/K$) is only defined \emph{up to a path} in $K$ and the relation between $u$ and $\omega$ is not one-to-one (see~\cite{TV2011} for a discussion on this subject).

These difficulties clear away if $K$ is a normal subgroup. Indeed, in that case, the coset manifold $G/K$ is a Lie group equipped with a right-invariant Riemannian metric. But this special case is not very useful for us, since $\Diff(S^{1})$ is \emph{simple}.

There is another situation where a right-invariant Riemannian metric on a homogeneous space reduces to an Euler-Arnold equation on a Lie group, namely when the right action of $G$ on the set of left cosets $G/K$ becomes \emph{simply transitive} when restricted to a subgroup $G_{0}$ of $G$.

This situation occurs for $\Diff(S^{1})/\Rot(S^{1})$, where $G_{0} := \Diff_{x_{0}}(S^{1})$, the subgroup of diffeomorphisms which fix one point, say $x_{0}$. Its Lie algebra $\mathrm{Vect}_{x_{0}}(S^{1})$ is the space of vector fields on the circle which vanish at $x_{0}$. The inertia operator $H\partial_{x}$ induces an isomorphism between $\mathrm{Vect}_{x_{0}}(S^{1})$ and the subspace of its topological dual defined by linear functionals
\begin{equation*}
  u \mapsto \int_{S^{1}} \omega(x)u(x)\,dx
\end{equation*}
where $\omega$ is smooth and has zero mean: $\int_{S^{1}} \omega(x)\,dx = 0$. We can therefore consider equation~\eqref{main} as an Euler-Arnold equation but on $\Diff_{x_{0}}(S^{1})$ rather than $\Diff(S^{1})$.

Even if $\ad(v)^{\top} u$ is not defined on $\mathrm{Vect}_{x_{0}}(S^{1})$ in that case because $u_{x}Hv_{x}$ does not have vanishing integral in general, the symmetric part of the bilinear operator $\ad(v)^{\top} u$ is however well defined and so is the geodesic spray. The proof of local existence of the geodesics and the fact that the exponential map is a local diffeomorphism follows the same line as in Section~\ref{subsec:local-well-posedness} (see~\cite{EKW2012,EK2014} for the details). Moreover, the proof of Lemma~\ref{indexformlemma} can be adapted in that case, without referring to $\Ad_{\eta}^{\top}$ which is not defined either.

To finish this section we will prove the following lemma.
\begin{lemma}
  The geodesic distance for the homogeneous $\dot{H}^{s}$-metric on $\Diff_{x_{0}}(S^{1})$ vanishes for $0 \le s \le 1/2$.
\end{lemma}

\begin{proof}
  The Fr\'{e}chet manifold $\Diff_{x_{0}}(S^{1})$ can be covered by a single chart, namely, the set of smooth, real $1$-periodic functions $\eta$ on $\RR$ such that $\eta(0) = 0$. Let $\eta_{0}$ and $\eta_{1}$ be two diffeomorphisms in $\Diff_{x_{0}}(S^{1})$. Given any path $\eta(t)$ in $\Diff(S^{1})$ with $\eta(0) = \eta_{0}$ and $\eta(1) = \eta_{1}$, we introduce
  \begin{equation*}
    L_{A}(\eta) = \int_{0}^{1} \norm{\partial_{t}\eta \circ \eta^{-1}}_{A} \, dt ,
  \end{equation*}
  where
  \begin{equation*}
    \norm{u}_{A}^{2} := \int_{S^{1}} (Au)u \, dx .
  \end{equation*}
  Let $\Lambda^{2s} := (1 + \abs{D}^{2})^{s}$ and $\dot{\Lambda}^{2s} := \abs{D}^{2s}$; see formula \eqref{fouriermultiplier}. We have $\norm{u}_{\dot{\Lambda}^{2s}} \le \norm{u}_{\Lambda^{2s}}$, for all $u \in \CS$ and thus
  \begin{equation}\label{eq:length}
    L_{\dot{\Lambda}^{2s}}(\eta) \le L_{\Lambda^{2s}}(\eta).
  \end{equation}
  However, to conclude that the geodesic distance $\mathrm{dist}^{\dot{H}^{s}}(\eta_{0},\eta_{1})$ for the $\dot{H}^{s}$-metric on $\Diff_{x_{0}}(S^{1})$ is bounded above by $\mathrm{dist}^{H^{s}}(\eta_{0},\eta_{1})$, the geodesic distance for the $H^{s}$-metric on $\Diff(S^{1})$, we need to show that the infimum of $L_{\dot{\Lambda}^{2s}}(\eta)$ on all paths $\eta$ in $\Diff(S^{1})$ joining $\eta_{0}$ and $\eta_{1}$ is the same as the infimum on paths $\eta$ in $\Diff_{x_{0}}(S^{1})$. To do so, given a path $\eta$ in $\Diff(S^{1})$, we introduce
  \begin{equation*}
    \tilde{\eta}(t,x) := \eta(t,x) - \eta(t,0)
  \end{equation*}
  which is a path in $\Diff_{x_{0}}(S^{1})$ and we check that
  \begin{equation*}
    \norm{\partial_{t}\tilde{\eta} \circ \tilde{\eta}^{-1}}_{\dot{\Lambda}^{2s}} = \norm{\partial_{t}\eta \circ \eta^{-1}}_{\dot{\Lambda}^{2s}}.
  \end{equation*}
  The conclusion follows since $\mathrm{dist}^{H^{s}}(\eta_{0},\eta_{1})$ vanishes for $0 \le s \le 1/2$.
\end{proof}

\section{The diffeomorphism group on the real line}
\label{appendix:real_line}

Shifting our focus from the diffeomorphism group of the compact manifold $S^{1}$ to the diffeomorphism group of the non-compact manifold $\RR$ requires us to specify certain decay conditions for the diffeomorphisms under consideration. The reason for this is that
the group of all orientation-preserving diffeomorphisms $\Diff(\RR)$ is not an open subset of $C^{\infty}(\RR,\RR)$ endowed with the compact $C^{\infty}$-topology. Thus it is not a smooth manifold with charts in the usual sense.
A method to overcome this difficulty is to restrict ourselves to subgroups of the whole diffeomorphism group, which are still smooth Fr\'{e}chet manifolds. This approach leads among others to the group of compactly supported diffeomorphisms $\Diff_{c}(\RR)$, the group of rapidly decaying diffeomorphisms $\Diff_{\mathcal S}(\RR)$, or the group of Sobolev diffeomorphisms $\Diff_{H^{\infty}}(\RR) = \cap_{k=2}^{\infty} \Diff_{H^k}(\RR)$:
\begin{align*}
  \Diff_{c}(\RR):          & =\left\{\varphi(x) = \id+f: f\in C_{c}^{\infty}(\RR),\, f'>-1 \right\}, \\
  \Diff_{\mathcal S}(\RR): & =\left\{\varphi(x) = \id+f: f\in \mathcal S(\RR),\, f'>-1 \right\},     \\
  \Diff_{H^{\infty}}(\RR): & =\left\{\varphi(x) = \id+f: f\in H^{\infty}(\RR),\, f'>-1 \right\}.     
\end{align*}
Note that for $\Diff(S^{1})$ all the above groups coincide.

\subsection{The full $H^{1/2}$-metric}

We will first focus on the full $H^{1/2}$-metric, which is geometrically the most straightforward of the three metrics
studied in this article. The $H^{1/2}$-scalar product yields a
well-defined metric on any of the above defined diffeomorphism groups. For analytic reasons we will consider the metric on the
group of all Sobolev diffeomorphisms. Analogous to Section~\ref{subsec:local-well-posedness}, let $\mathcal{D}^{q}(\RR)$ be the Hilbert approximation manifold of $\Diff_{H^{\infty}}(\RR)$. It has been shown in~\cite{BEK2015} that the $H^{1/2}$-metric extends to
a smooth Riemannian metric on $\mathcal{D}^{q}(\RR)$ for any $q>\frac32$.
Thus we obtain the analogue of the local well-posedness statement Theorem~\ref{thm:local-wellposedness} for the full $H^{1/2}$-metric
on both $\Diff_{H^{\infty}}(\RR)$ and $\mathcal{D}^{q}(\RR)$:
\begin{theorem}
  For any $q>\frac32$ the exponential map of the $H^{1/2}$-metric is a smooth local diffeomorphism from a neighborhood $V\subset H^{q}(\RR)$ of $0$ onto a neighborhood $U$ of the identity in $\D{q}(\RR)$.
  This result continues to hold in the smooth category $\Diff_{H^{\infty}}(\RR)$.
\end{theorem}
For integer order metrics, one can show a similar statement on the group of compactly supported diffeomorphisms, using groups of $C^k$-diffeomorphisms as approximation spaces. We are
not aware of any generalization of this result to the case of fractional-order metrics.

As a next step we want to discuss the corresponding geodesic distance. There are, however, two obstacles against generalizing the proof for $\Diff(S^{1})$ presented in Section~\ref{sec:geodesic_distance}.
First, the group of Sobolev diffeomorphisms $\Diff_{H^{\infty}}(\RR)$ is not a simple group, e.g., $\Diff_{c}(\RR)$
and $\Diff_{\mathcal S}(\RR)$ are normal subgroups. Furthermore the shift $\varphi(x)=x+1$ is not an element of the group. For the $H^{s}$-metric of order $s<\frac12$ a direct construction for the vanishing geodesic distance result --
using neither the simplicity nor the shift diffeomorphism --- has been proven in~\cite{BBHM2013}. This construction uses indicator function, which are elements of $H^{s}(\RR)$ for $s<\frac12$, but not of $H^{1/2}$. Although we strongly believe that the result continues to hold for the $H^{1/2}$-metric on the diffeomorphism group of the real line, this question remains open for the time being.

Finally we want to comment on the non-Fredholmness results, presented in Section~\ref{sec:Fredholmness}. The constant vector field $u=1$ is not an element of the Lie-algebra in the non-compact situation. Thus we do not have the simple steady state
solutions that were the main ingredient for the direct proof of the non-Fredholmness. On the other hand, the construction of conjugate points along arbitrary geodesic presented in Section~\ref{sec:conjugate_{a}rbitrary} continues to hold in this situation, and thus we obtain the non-Fredholmness of the exponential mapping also in the non-periodic case.

\subsection{The degenerate $\dot{H}^{1/2}$-metric and the $\mu H^{1/2}$-metric.}

We will finish this article with some comments on the degenerate $\dot{H}^{1/2}$-metric and the $\mu H^{1/2}$-metric.
To prove vanishing geodesic distance for these two metrics, one has to overcome similar problems as for the
full $H^{1/2}$-metric on $\Diff(\RR)$ and thus we will not further consider this question. Instead we will sketch how one could prove a well-posedness statement for the geodesic equation of these two metrics.

The $\mu H^{1/2}$-metric does not induce a bounded inner product on ${H^{\infty}}(\RR)$ and its Sobolev approximations $H^{q}(\RR)$.
Instead of these $L^{2}$-type decay conditions, we have to consider $L^{1}$-type decay condition, yielding the Lie-group
$\Diff_{W^{\infty,1}}(\RR)$ and its Banach approximations $\mathcal{D}^{q,1}(\RR)$.
Here $W^{\infty,1}$ denotes the intersection of all Besov spaces $W^{k,1}$ of type $L^{1}$.
For a proof of the Lie-group structure of $\Diff_{W^{\infty,1}}(\RR)$ we refer to the article~\cite{BBM2014}.
On $\Diff_{W^{\infty,1}}(\RR)$ the $\mu H^{1/2}$-metric induces a smooth right-invariant metric. To prove smoothness of the corresponding exponential mapping it remains to prove that the metric extends to a smooth metric on its Banach approximations $\mathcal{D}^{q,1}(\RR)$ for high enough $q$. This should follow using similar methods as in~\cite{EK2014,BEK2015}.

The situation for the degenerate $\dot{H}^{1/2}$-metric is even more interesting. This inner product defines a well-defined and non-degenerate metric for any of the previously diffeomorphism groups. However it turns out that neither the adjoint $\ad^\top_u v$ nor its symmetrized version exist on any of these groups. As a consequence these groups do not admit a geodesic equation (or Levi-Civita covariant derivative) for these metrics. This phenomenon was first observed for the related situation of the $\dot{H}^{1}$-metric in the article~\cite{BBM2014}.

Extending these groups by a one-dimensional space (to include diffeomorphisms that are a shift near positive infinity) yields groups
where the symmetrized version of the adjoint $\frac12(\ad^\top_u v+\ad^\top_{v} u)$ is defined. Thus we obtain a meaningful geodesic equation. To prove smoothness of the corresponding exponential mapping it remains again
to prove that the metric extends to a smooth metric on the corresponding Banach approximation spaces, which follows similarily as in~\cite{EK2014,BEK2015}.

Finally we consider blowup results for the $\dot{H}^{1/2}$-metric (equivalently, for the $\mu H^{1/2}$-metric with initial
mean zero). This was the situation originally studied in Castro-C\'{o}rdoba~\cite{CC2010}, where blowup was proved for odd
initial data. Here we extend this result by adapting Lemma~\ref{signlemma}, which works essentially the same way as in
the periodic case and avoids the machinery of Mellin transforms.

\begin{lemma}\label{signlemmaline}
  Suppose $f\colon \RR\to \RR$ is a nonconstant function, and let $\Lambda$ be the Fourier multiplier defined so that $\widehat{\Lambda f}(\xi) = \abs{\xi} \hat{f}(\xi)$.
  Define $g_{p} = H(fH\Lambda^{p}f) + f\Lambda^{p}f$;
  then $g_{p}(x)>0$ for every $x\in \RR$.
\end{lemma}

\begin{proof}
  Again it is sufficient to prove this at $x=0$, by translation invariance. Define the Fourier transform $\mathcal{F}$ as
  \begin{equation*}
    \mathcal{F}(f)(\xi) = \hat{f}(\xi) := \int_{\RR} f(x) e^{-2\pi ix\xi}\,dx
  \end{equation*}
  so that $\mathcal{F}^{-1}(f)(\xi) = \mathcal{F}(f)(-\xi)$ and note that $\hat{f}(-\xi) = \overline{\hat{f}(\xi)}$ since $f$ is real-valued. The same computation as in Theorem~\ref{signlemma} (which also appears in~\cite{CC2010}) gives
  \begin{equation*}
    H(fH\Lambda^{p}f)(0) = - \int_{-\infty}^{\infty} \int_{-\infty}^{\infty} \sgn{(\xi+\psi)} \sgn{(\psi)} \abs{\psi}^{p} \hat{f}(\xi) \hat{f}(\psi) \, d\psi\, d\xi.
  \end{equation*}
  We therefore have
  \begin{equation*}
    g_{p}(0) = \int_{-\infty}^{\infty} \int_{-\infty}^{\infty} \Big[ 1 - \sgn{(\xi+\psi)} \sgn{(\psi)}\Big] \abs{\psi}^{p} \hat{f}(\xi) \hat{f}(\psi)\,d\psi\,d\xi,
  \end{equation*}
  which may be rewritten, using the same techniques as in Theorem~\ref{signlemma}, as
  \begin{equation*}
    g_{p}(0) = 2\int_{0}^{\infty} p\zeta^{p-1} \abs{\phi(\zeta)}^{2}\,d\zeta, \quad \text{where} \quad \phi(\zeta) = \int_{\zeta}^{\infty} \hat{f}(\xi)\,d\xi .
  \end{equation*}
\end{proof}



\begin{thebibliography}{10}
  
  \bibitem{Arn1966}
  V.~I. Arnold.
  \newblock {S}ur la g\'{e}om\'{e}trie diff\'{e}rentielle des groupes de {L}ie de
  dimension infinie et ses applications \`{a} l'hydrodynamique des fluides
  parfaits.
  \newblock {\em Ann. Inst. Fourier (Grenoble)}, 16(fasc. 1):319--361, 1966.
  
  \bibitem{AK1998}
  V.~I. Arnold and B.~Khesin.
  \newblock {\em {T}opological {M}ethods in {H}ydrodynamics}, volume 125 of {\em
  Applied Mathematical Sciences}.
  \newblock Springer-Verlag, New York, 1998.
  
  \bibitem{BBHM2013}
  M.~Bauer, M.~Bruveris, P.~Harms, and P.~W. Michor.
  \newblock {G}eodesic distance for right invariant {S}obolev metrics of
  fractional order on the diffeomorphism group.
  \newblock {\em Ann. Global Anal. Geom.}, 44(1):5--21, 2013.
  
  \bibitem{BBM2013}
  M.~Bauer, M.~Bruveris, and P.~W. Michor.
  \newblock {G}eodesic distance for right invariant {S}obolev metrics of
  fractional order on the diffeomorphism group. {II}.
  \newblock {\em Ann. Global Anal. Geom.}, 44(4):361--368, 2013.
  
  \bibitem{BBM2014}
  M.~Bauer, M.~Bruveris, and P.~W. Michor.
  \newblock {H}omogeneous {S}obolev metric of order one on diffeomorphism groups
  on real line.
  \newblock {\em J. Nonlinear Sci.}, 24(5):769--808, 2014.
  
  \bibitem{BEK2015}
  M.~Bauer, J.~Escher, and B.~Kolev.
  \newblock {L}ocal and global well-posedness of the fractional order {EPD}iff
  equation on $\mathbb {R}^d$.
  \newblock {\em J. Differential Equations}, 258(6):2010--2053, 2015.
  
  \bibitem{BKM1984}
  J.~T. Beale, T.~Kato, and A.~Majda.
  \newblock {R}emarks on the breakdown of smooth solutions for the {$3$}-{D}
  {E}uler equations.
  \newblock {\em Comm. Math. Phys.}, 94(1):61--66, 1984.
  
  \bibitem{BR1987}
  M.~J.~Bowick and S.~G.~Rajeev.
  \newblock {T}he holomorphic geometry of closed bosonic string theory and $\Diff(S^1)/S^1$.
  \newblock {\em Nuclear Phys.}, B293(2):348--384, 1987.
  
  \bibitem{BCZ2015}
  A.~Bressan, G.~Chen, and Q.~Zhang.
  \newblock Uniqueness of conservative solutions to the {C}amassa-{H}olm equation
  via characteristics.
  \newblock {\em Discrete Contin. Dyn. Syst.}, 35(1):25--42, 2015.
  
  \bibitem{BC2007}
  A.~Bressan and A.~Constantin.
  \newblock Global conservative solutions of the {C}amassa-{H}olm equation.
  \newblock {\em Arch. Ration. Mech. Anal.}, 183(2):215--239, 2007.
  
  \bibitem{CC2010}
  A.~Castro and D.~C\'{o}rdoba.
  \newblock {I}nfinite energy solutions of the surface quasi-geostrophic
  equation.
  \newblock {\em Adv. Math.}, 225(4):1820--1829, 2010.
  
  \bibitem{CHK+2014}
  K.~Choi, T.~Y. Hou, A.~Kiselev, G.~Luo, V.~Sverak, and Y.~Yao.
  \newblock {O}n the finite-time blowup of a 1{D} model for the 3{D} axisymmetric
  {E}uler equations.
  \newblock {\em arXiv:1407.4776}, 2014.
  
  \bibitem{CLM1985}
  P.~Constantin, P.~D. Lax, and A.~Majda.
  \newblock {A} simple one-dimensional model for the three-dimensional vorticity
  equation.
  \newblock {\em Comm. Pure Appl. Math.}, 38(6):715--724, 1985.
  
  \bibitem{CC2003}
  A.~C\'{o}rdoba and D.~C\'{o}rdoba.
  \newblock {A} pointwise estimate for fractionary derivatives with applications
  to partial differential equations.
  \newblock {\em Proc. Natl. Acad. Sci. USA}, 100(26):15316--15317, 2003.
  
  \bibitem{CM2015}
  A.~C\'{o}rdoba and \'{A}.~D. Mart{\'i}nez.
  \newblock {A} pointwise inequality for fractional {L}aplacians.
  \newblock {\em arXiv:1502.01635}, 2015.
  
  \bibitem{Ebi1984}
  D.~G. Ebin.
  \newblock {A} concise presentation of the {E}uler equations of hydrodynamics.
  \newblock {\em Comm. Partial Differential Equations}, 9(6):539--559, 1984.
  
  \bibitem{EM1970}
  D.~G. Ebin and J.~E. Marsden.
  \newblock {G}roups of diffeomorphisms and the motion of an incompressible
  fluid.
  \newblock {\em Ann. of Math. (2)}, 92:102--163, 1970.
  
  \bibitem{EMP2006}
  D.~G. Ebin, G.~Misio{\l}ek, and S.~C. Preston.
  \newblock {S}ingularities of the exponential map on the volume-preserving
  diffeomorphism group.
  \newblock {\em Geom. Funct. Anal.}, 16(4):850--868, 2006.
  
  \bibitem{ER1991}
  Y.~Eliashberg and T.~Ratiu.
  \newblock The diameter of the symplectomorphism group is infinite.
  \newblock {\em Invent. Math.}, 103(2):327--340, 1991.
  
  \bibitem{EG1976}
  C.~Eliezer and A.~Gray.
  \newblock {A} note on the time-dependent harmonic oscillator.
  \newblock {\em Siam J. Appl. Math.}, 30(3):463--468, 1976.
  
  \bibitem{EK2014}
  J.~Escher and B.~Kolev.
  \newblock {R}ight-invariant {S}obolev metrics of fractional order on the
  diffeomorphism group of the circle.
  \newblock {\em J. Geom. Mech.}, 6(3):335--372, 2014.
  
  \bibitem{EK2014a}
  J.~Escher and B.~Kolev.
  \newblock {G}eodesic completeness for {S}obolev {$H\sp s$}-metrics on the
  diffeomorphism group of the circle.
  \newblock {\em J. Evol. Equ.}, 14(4-5):949--968, 2014.
  
  \bibitem{EKW2012}
  J.~Escher, B.~Kolev, and M.~Wunsch.
  \newblock {T}he geometry of a vorticity model equation.
  \newblock {\em Commun. Pure Appl. Anal.}, 11(4):1407--1419, 2012.
  
  \bibitem{Fre1988}
  D.~S.~Freed.
  \newblock {T}he geometry of loop groups.
  \newblock {\em J. Differential Geometry}, 28(2):223--276, 1988.
  
  \bibitem{GBR2015}
  F.~Gay-Balmaz and T.~Ratiu.
  \newblock The geometry of the universal Teichm\"uller space and the Euler-Weil-Petersson equation.
  \newblock {\em Adv. Math.}, 279:717--778, 2015.
  
  \bibitem{GL2006}
  M.~Gordina and P.~Lescot.
  \newblock {R}iemannian geometry of $\Diff(S^1)/S^1$.
  \newblock {\em J. Funct. Anal.}, 239:611--630, 2006.
  
  \bibitem{HTV2002}
  S.~Haller, J.~Teichmann, and C.~Vizman.
  \newblock {T}otally geodesic subgroups of diffeomorphisms.
  \newblock {\em J. Geom. Phys.}, 42(4):342--354, 2002.
  
  \bibitem{IKT2013}
  H.~Inci, T.~Kappeler, and P.~Topalov.
  \newblock {\em {O}n the {R}egularity of the {C}omposition of
    {D}iffeomorphisms}, volume 226 of {\em Memoirs of the American Mathematical
  Society}, no. 1062.
  \newblock American Mathematical Society, 2013.
  
  \bibitem{KLM2008}
  B.~Khesin, J.~Lenells, and G.~Misio{\l}ek.
  \newblock {G}eneralized {H}unter-{S}axton equation and the geometry of the
  group of circle diffeomorphisms.
  \newblock {\em Math. Ann.}, 342(3):617--656, 2008.
  
  \bibitem{KY1987}
  A.~A.~Kirillov and D.~V.~Yur'ev.
  \newblock {K}\"ahler geometry of the infinite-dimensional homogeneous space $M=\Diff_+(S^1)/\Rot(S^1)$.
  \newblock {\em Functional Anal. Appl.}, 21(4):284--294, 1987.
  
  \bibitem{LA2008}
  P.~Leach and K.~Andriopoulos.
  \newblock {T}he {E}rmakov equation: a commentary.
  \newblock {\em Appl. Anal. Discrete Math.}, 2(2):146--157, 2008.
  
  \bibitem{Len2007}
  J.~Lenells.
  \newblock {T}he {H}unter-{S}axton equation describes the geodesic flow on a
  sphere.
  \newblock {\em J. Geom. Phys.}, 57(10):2049--2064, 2007.
  
  \bibitem{MB2002}
  A.~J.~Majda and A.~L.~Bertozzi.
  \newblock {\em {V}orticity and {I}ncompressible {F}low}, volume~27 of {\em
  Cambridge Texts in Applied Mathematics}.
  \newblock Cambridge University Press, Cambridge, 2002.
  
  \bibitem{McK2003}
  H.~P.~McKean.
  \newblock {F}redholm determinants and the {C}amassa-{H}olm hierarchy.
  \newblock {\em Comm. Pure Appl. Math.}, 56(5):638--680, 2003.
  
  \bibitem{MM2005}
  P.~W.~Michor and D.~Mumford.
  \newblock {V}anishing geodesic distance on spaces of submanifolds and
  diffeomorphisms.
  \newblock {\em Doc. Math.}, 10:217--245, 2005.
  
  \bibitem{MP2010}
  G.~Misio{\l}ek and S.~C.~Preston.
  \newblock {F}redholm properties of {R}iemannian exponential maps on
  diffeomorphism groups.
  \newblock {\em Invent. Math.}, 179(1):191--227, 2010.
  
  \bibitem{Mit1970}
  D.~S.~Mitrinovi{\'c}.
  \newblock {\em {A}nalytic inequalities}, volume 165 of {\em Die Grundlehren der mathematischen Wissenschaften}.
  \newblock Springer-Verlag, New York-Berlin, 1970.
  
  \bibitem{NJ1990}
  S.~Nag and A.~Verjovsky.
  \newblock $\Diff(S^1)$ and the Teichm\"uller spaces.
  \newblock {\em Comm. Math. Phys.}, 130(1):123--138, 1990.
  
  \bibitem{OSW2008}
  H.~Okamoto, T.~Sakajo, and M.~Wunsch.
  \newblock {O}n a generalization of the {C}onstantin-{L}ax-{M}ajda equation.
  \newblock {\em Nonlinearity}, 21(10):2447--2461, 2008.
  
  \bibitem{Poi1901}
  H.~Poincar\'{e}.
  \newblock {S}ur une forme nouvelle des \'{e}quations de la m\'{e}canique.
  \newblock {\em C.R. Acad. Sci.}, 132:369--371, 1901.
  
  \bibitem{Pre2006}
  S.~C.~Preston.
  \newblock {O}n the volumorphism group, the first conjugate point is always the
  hardest.
  \newblock {\em Comm. Math. Phys.}, 267(2):493--513, 2006.
  
  \bibitem{Pre2010}
  S.~C.~Preston.
  \newblock {A} geometric rigidity theorem for hydrodynamical blowup.
  \newblock {\em Comm. Partial Differential Equations}, 35(11):2007--2020, 2010.
  
  \bibitem{PS2015}
  S.~C.~Preston and A.~Sarria.
  \newblock {L}agrangian aspects of the axisymmetric {E}uler equation.
  \newblock {\em arXiv:1505.05569}, 2015.
  
  \bibitem{Ser2010}
  A.~Sergeev.
  \newblock {\em {K}\"ahler {G}eometry of {L}oop {S}paces}, volume 23 of {\em MSJ Memoirs}.
  \newblock Mathematical Society of Japan, Tokyo, 2010.
  
  \bibitem{Shn1994}
  A.~Shnirelman.
  \newblock {G}eneralized fluid flows, their approximation and applications.
  \newblock {\em Geom. Funct. Analysis}, 4(5):586--620, 1994.
  
  \bibitem{TV2011}
  F.~Ti\u{g}lay and C.~Vizman.
  \newblock {G}eneralized {E}uler-{P}oincar\'{e} equations on {L}ie groups and
  homogeneous spaces, orbit invariants and applications.
  \newblock {\em Lett. Math. Phys.}, 97(1):45--60, 2011.
  
  \bibitem{Tri2001}
  H.~Triebel.
  \newblock {\em {T}he {S}tructure of {F}unctions}, volume~97 of {\em Monographs
  in Mathematics}.
  \newblock Birkh\"{a}user Verlag, Basel, 2001.
  
  \bibitem{Wun2010}
  M.~Wunsch.
  \newblock {O}n the geodesic flow on the group of diffeomorphisms of the circle
  with a fractional {S}obolev right-invariant metric.
  \newblock {\em J. Nonlinear Math. Phys.}, 17(1):7--11, 2010.
  
  \bibitem{Wun2011}
  M.~Wunsch.
  \newblock {T}he generalized {C}onstantin-{L}ax-{M}ajda equation revisited.
  \newblock {\em Commun. Math. Sci.}, 9(3):929--936, 2011.
  
  \bibitem{Zum1988}
  B.~Zumino.
  \newblock {T}he geometry of the {V}irasoro group for physicists.
  \newblock {\em Particle Physics, NATO ASI Series}, 173:81--98, 1988.
  
\end{thebibliography}
\end{document}